\documentclass[12pt]{amsart}

\usepackage[utf8]{inputenc}  
\usepackage[T1]{fontenc}

\usepackage{amsmath}
\usepackage{amsfonts}
\usepackage{amssymb}
\usepackage{amsthm}       
\usepackage{bbold}			
\usepackage{hyperref}      

\usepackage{geometry}
\usepackage{tikz} 

\renewcommand {\phi} {\varphi}
\renewcommand{\epsilon}{\varepsilon}
\newcommand {\abs} [1] {\left\vert #1 \right\vert}

\newcommand {\ceil} [1] {\left\lceil #1 \right\rceil}
\newcommand {\floor} [1] {\left\lfloor #1 \right\rfloor}

\newcommand {\Prod} {\prod\limits}
\newcommand {\Sum} {\sum\limits}
\newcommand {\Lim} {\lim\limits}
\newcommand {\Limsup} {\limsup\limits}
\newcommand {\Liminf} {\liminf\limits}

\renewcommand{\P}{\mathbb{P}}
\newcommand{\E}{\mathbb{E}}
\newcommand{\N}{\mathbb{N}}
\newcommand{\Z}{\mathbb{Z}}
\newcommand{\Q}{\mathbb{Q}}
\newcommand{\R}{\mathbb{R}}

\newcommand{\1}{\mathbb{1}}

\renewcommand{\(}{\left(}
\renewcommand{\)}{\right)}

\renewcommand{\geq}{\geqslant}
\renewcommand{\leq}{\leqslant}

\newtheorem{de}{Definition} 
\newtheorem{theo}{Theorem}    
\newtheorem{prop}[theo]{Proposition}   
\newtheorem{corollary}[theo]{Corollary} 
\newtheorem{lemma}{Lemma}
\newtheorem{rmq}{Remark}

\numberwithin{equation}{section}     

\title[ ]{Sub-ballistic random walk in Dirichlet environment}
\thanks{This work was supported by the ANR project MEMEMO}

\author[ ]{\'Elodie Bouchet}
\address{Universit\'e de Lyon, Universit\'e Lyon 1,
Institut Camille Jordan, CNRS UMR 5208, 43, Boulevard du 11 novembre 1918,
69622 Villeurbanne Cedex, France} \email{bouchet@math.univ-lyon1.fr}

\begin{document}

\begin{abstract}

We consider random walks in Dirichlet environment (RWDE) on $\Z ^d$, for $ d \geq 3 $, in the sub-ballistic case. We associate to any parameter $ (\alpha_1, \dots, \alpha _{2d}) $ of the Dirichlet law a time-change to accelerate the walk. We prove that the continuous-time accelerated walk has an absolutely continuous invariant probability measure for the environment viewed from the particle. This allows to characterize directional transience for the initial RWDE. It solves as a corollary the problem of Kalikow's $0-1$ law in the Dirichlet case in any dimension. Furthermore, we find the polynomial order of the magnitude of the original walk's displacement.

\end{abstract}

\maketitle

\section{Introduction}

The behaviour of random walks in random environment (RWRE) is fairly well understood in the case of dimension 1 (see Solomon (\cite{Solomon}), Kesten, Kozlov, Spitzer (\cite{KKS}) and Sinaï(\cite{Sinai})). In the multidimensional case, some results are available under ballisticity conditions (we refer to \cite{Zeitouni} and \cite{Sz} for an overview of progress in this direction), or in the case of small perturbations. But some simple questions remain unanswered. For example, there is no general characterization of recurrence, Kalikow's $0-1$ law is known only for $d \leq 2$ (\cite{ZM}).

Random walks in Dirichlet environment (RWDE) is the special case when the transition probabilities at each site are chosen as i.i.d. Dirichlet random variables. RWDE are interesting because of the analytical simplifications they offer, and because of their link with reinforced random walks. Indeed, the annealed law of a RWDE corresponds to the law of a linearly directed-edge reinforced random walk (\cite{ES}, \cite{Pemantle}). This model first appeared in \cite{Pemantle} in relation with edge reinforced random walks on trees. It was then studied on $ \Z \times G $ (\cite{KeaneRolles}), and on $ \Z ^d $ (\cite{ES2},\cite{Tournier},\cite{S2},\cite{S},\cite{ST}). 

We are interested in RWDE on $\Z ^d$ for $d \geq 3$. A condition on the weights ensures that the mean time spent in finite boxes is finite. Under this condition, it was proved (\cite{S}) that there exists an invariant probability measure for the environment viewed from the particle, absolutely continuous with respect to the law of the environment. Using \cite{ST}, this gives some criteria on ballisticity.

In this paper, we focus on the case when the condition on the weights is not satisfied. Then the mean time spent in finite boxes is infinite, and there is no absolutely continuous invariant probability measure (\cite{S}). The law of large numbers gives a zero speed. To overcome this difficulty, we construct a time-change that accelerates the walk, such that the accelerated walk spends a finite mean time in finite boxes. An absolutely continuous invariant probability measure then exists. With ergodic results, it gives a characterization of the directional recurrence in the sub-ballistic case. As a corollary, it solves the problem of Kalikow's $0-1$ law in the Dirichlet case (the case $d=2$ has been treated in \cite{ZM}).

Besides, in the directionally transient case, we show a law of large numbers with positive speed for our accelerated walk. This gives the polynomial order of the magnitude of the original walk's displacement, and could be a first step towards a limit theorem for the original RWDE.

\section{Definitions and statement of the results}

Let $(e_1,\dots,e_d)$ be the canonical base of $\Z^d$, $ d \geq 3 $, and set $e_j = - e_{j-d}$, for $j \in [\![ d+1,2d ]\!]$. The set $\{ e_1, \dots ,e_{2d} \}$ is the set of unit vectors of $\Z^d$. We denote by $\Vert z \Vert = \sum _{i=1} ^d |z_i|$ the $L_1$-norm of $z \in \Z^d$, and write $x \sim y$ if $ \Vert y - x \Vert = 1$. We consider the set of directed edges $E = \{ (x,y) \in (\Z^d)^2, x \sim y \}$. Let $\Omega$ be the set of all possible environments on $\Z ^d$ :
\[ \Omega =  \{ \omega = (\omega(x,y)) _{ x \sim y} \in ] 0 , 1] ^E \text{ such that } \forall x \in \Z ^d, \sum _{i=1}^{2d} \omega (x , x+ e_i)= 1  \} .\]
For each $ \omega  \in \Omega$, we run a Markov chain $Z_n$ on $ \Z ^d $ defined by the following transition probabilities : 
$ \forall (x,y) \in \Z^d$, $ \forall i \in [\![ 1,2d ]\!]$, 
\[P _x ^\omega (Z_{n+1} = y + e_i | Z_n = y ) = \omega (y, y + e_i) .\]

We are interested in random iid Dirichlet environments. Given a family of positive weights $ (\alpha _1, \dots , \alpha _{2d}) $, a random iid Dirichlet environment is $\omega  \in \Omega$ constructed by choosing independently at each site $x \in \Z^d$ the values of $ \( \omega ( x, x+e_i ) \) _{i \in [\![1,2d]\!]} $ according to a Dirichlet law with parameters $ (\alpha _1, \dots , \alpha _{2d}) $ that is with density : 
\[ \frac{ \Gamma \( \sum _{i=1} ^{2d} \alpha _i \) }{ \prod_{i=1} ^{2d} \Gamma \( \alpha _i \) } \( \prod _{i=1} ^{2d} x_i ^{\alpha _i - 1} \) dx_1 \dots dx_{2d-1} \]
on the simplex 
\[ \{ (x_1,\dots, x_{2d}) \in ]0,1]^{2d}, \sum _{i=1} ^{2d} x_i = 1 \} .\]
Here $\Gamma$ denotes the Gamma function $ \Gamma (\alpha) = \int _0 ^\infty t ^{\alpha - 1} e ^{-t} dt$ , and $dx_1 \dots dx_{2d-1}$ represents the image of the Lebesgue measure on $ \R^ {2d-1} $ by the application $( x_1, \dots ,x_{2d-1} )  \to ( x_1, \dots , x_{2d-1}, 1- x_1 - \dots - x_{2d-1} ) $. Obviously, the law does not depend on the specific role of $ x_{2d}$. We denote by $ \P ^{(\alpha)} $ the law obtained on $\Omega$ this way, by $ \E ^{(\alpha)} $ the expectation with respect to $\P ^{(\alpha)}$, and by $ \P_x ^{(\alpha)}[.] = \E ^{(\alpha)} [ P _x ^\omega(.) ] $ the annealed law of the process starting at $x$.

In \cite{S}, it was proved that when 
\[ \kappa =  2 \( \sum_{i=1} ^{2d} \alpha _i \) - \max _{i=1,\dots, d} (\alpha _i + \alpha _{i+d}) > 1 ,\]
there exists an invariant probability measure for the environment viewed from the particle, absolutely continuous with respect to $ \P ^{(\alpha)} $. This leads to a complete description of ballistic regimes and directional transience. However, when $ \kappa \leq 1$, such an invariant probability does not exist, and we only know that the walk is sub-ballistic. In this paper, we focus on the case $\kappa \leq 1$. We prove the existence of an invariant probability measure for an accelerated walk. This allows to characterize recurrence in each direction for the initial walk.

Let $ \sigma = (e ^1, \dots, e ^n) $ be a directed path. By directed path, we mean a sequence of directed edges $e^i$ such that $ \overline{e^i} = \underline{ e ^{i+1}} $ for all $i$ ($\underline{e}$ and $\overline{e}$ are the head and tail of the edge $e$). We note $ \omega _\sigma  = \prod _{i=1} ^n \omega (e^i) $. Let $ \Lambda $ be a finite connected set of vertices containing $0$. Our accelerating function is $\gamma ^\omega (x) = \frac{1}{\sum \omega _\sigma}$, where the sum is on all $\sigma$ finite simple (each vertex is visited at most once) paths starting from $x$, going out of $x + \Lambda$, and stopped just after exiting $x+\Lambda$. Let $X_t$ be the continuous-time Markov chain whose jump rate from $x$ to $y$ is $ \gamma ^\omega (x)  \omega (x,y) $, with $X_0 = 0$. Then $ Z_n = X _{t_n} $, for $ t_n = \sum _{k=1} ^n \frac{1}{\gamma ^\omega (Z_k) } E_k$, where the $E_i$ are independent exponentially distributed random variables with rate parameters $1$ : $X_t$ is an accelerated version of the walk $Z_n$.

We note $ ( \tau _x ) _{x \in \Z ^d} $ the shift on the environment defined by : $ \tau _x \omega (y,z) = \omega (x+y,x+z) $, and call process seen from the particle the process defined by $ \overline{\omega} _t = \tau _{X_t} \omega $. Under $ P _0 ^{\omega _0} $ ($ \omega _0 \in \Omega $), $\overline{\omega} _t$ is a Markov process on state space $ \Omega $, his generator $ R $ is given by
\[ R f (\omega) = \sum _{i=0} ^{2d} \gamma ^\omega (0) \omega (0, e_i) f(\tau _{e_i} \omega ) ,\]
for all bounded measurable functions $f$ on $\Omega$. Invariant probability measures absolutely continuous with respect to the law of the environment are a classical tool to study processes viewed from the particle. The following theorem provides one for our accelerated walk.

\begin{theo}\label{th1}

Let $d \geq 3$ and $ \P ^{(\alpha)} $ be the law of the Dirichlet environment for the weights $ (\alpha _1, \dots , \alpha_{2d}) $. Let $\kappa ^\Lambda > 0$ be defined by
\[ \kappa ^\Lambda = \min \{ \sum _{e \in \partial_+(K)}  \alpha _e \; , \; K \text{ connected set of vertices }, 0 \in K \text{ and } \partial \Lambda \cap K \neq \emptyset \} \]
where  $\partial_+(K) = \{ e \in E, \; \underline{e} \in K , \; \overline{e} \notin K \}$ and $ \partial \Lambda = \{ x \in \Lambda | \exists y \sim x \text{ such that } y \notin \Lambda \} $.
If $\kappa ^\Lambda > 1$, there exists a unique probability measure $ \Q ^{(\alpha)} $ on $ \Omega $ that is absolutely continuous with respect to $ \P ^{(\alpha)} $ and invariant for the generator $R$. Furthermore, $ \frac{ d \Q^{(\alpha)} }{d \P^{(\alpha)}} $ is in $ L_p (\P ^{(\alpha)}) $ for all $ 1 \leq p < \kappa ^\Lambda $.

\end{theo}

\begin{center} 
\begin{tikzpicture}[scale=1,>=stealth]
\draw [ultra thin, gray] (-2,-4) grid (6,3);
\tikzstyle{fleche}=[thick, dashed, ->] 
\draw[fleche] (0,0) -- (-1,0) ;
\draw[fleche] (0,0) -- (0,1) ;
\draw[fleche] (0,0) -- (0,-1) ;
\draw[fleche] (2,0) -- (3,0) ;
\draw[fleche] (1,0) -- (1,1) ;
\draw[fleche] (1,-1) -- (0,-1) ;
\draw[fleche] (1,-1) -- (1,-2) ;
\draw[fleche] (2,0) -- (2,1) ;
\draw[fleche] (2,-1) -- (3,-1) ;
\draw[fleche] (2,-2) -- (1,-2) ;
\draw[fleche] (2,-2) -- (2,-3) ;
\draw[fleche] (3,-2) -- (3,-1) ;
\draw[fleche] (3,-2) -- (3,-3) ;
\draw[fleche] (4,-2) -- (4,-1) ;
\draw[fleche] (4,-2) -- (4,-3) ;
\draw[fleche] (4,-2) -- (5,-2) ;
\draw[ -, very thick] (0,0) -- (1,0) -- (1,-1) -- (2,-1) -- (2,-2) -- (3, -2) -- (4,-2);
\draw[ -, very thick] (1,0) -- (2,0) -- (2,-1);
\draw [thick, dotted] (-1,2) -- (3,2) -- (3,1) -- (4,0) -- (4,-3) -- (-1,-3) -- (-1,2) ;
\draw (0,0) node[above right]{$0$} ;
\draw (3,2) node[above right]{$\Lambda$};
\end{tikzpicture}
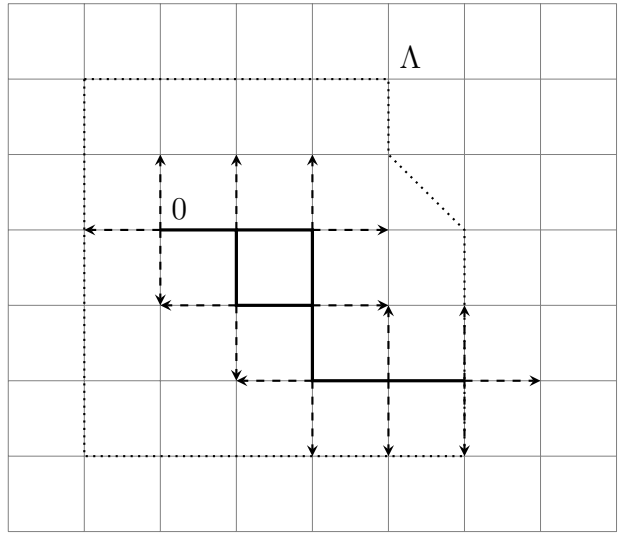
\\Figure : $ \Sum _{e \in \partial_+(K)}  \alpha _e $ (dashed arrows) for an arbitrary $K$ (thick lines).
\end{center}

\begin{rmq}
If $\Lambda$ is a box of radius $R _\Lambda$, the formula is explicit :
\[ \kappa ^\Lambda =  \min _{i_0 \in [\![ 1,d ]\!]} \(  \alpha _{i_0} + \alpha _{i_0 + d} + (R_\Lambda + 1) \sum _{i \neq i_0} (\alpha _{i} + \alpha _{i + d})  \) .\]
\end{rmq}

\begin{rmq}
$ \kappa ^\Lambda $ can be made as big as we want by taking the set $\Lambda$ big enough. Then for each $ (\alpha _1, \dots , \alpha _{2d} ) $, there exists an acceleration function such that the accelerated walk verifies theorem \ref{th1}.
\end{rmq}

Let $ d_{\alpha} = \mathbb{E}_0 ^{(\alpha)} [Z_1] = \frac{1}{\sum _{i = 1} ^{2d} \alpha _i} \sum _{i=1} ^{2d} \alpha _i e_i $ be the drift after the first jump.

\begin{theo}\label{thppal}
Let $d \geq 3$,

\begin{itemize}
\item[i)] If $ \kappa ^\Lambda > 1 $ and $ d _\alpha = 0 $, then 
\[ \lim _{ t \to +\infty } \frac{X_t}{t} = 0 , \; \P _0 ^{(\alpha)} \text{a.s.} ,\]
and $ \forall i = 1 \dots d $, 
\[ \liminf _{ t \to +\infty } X_t \cdot e_i = - \infty  , \;  \limsup _{ t \to +\infty } X_t \cdot e_i = + \infty , \; \P _0 ^{(\alpha)} \text{a.s.}. \]

\item[ii)] If $ \kappa ^\Lambda > 1 $ and $ d _\alpha \neq 0 $, then $ \exists v \neq 0 $ such that 
\[ \lim _{ t \to +\infty } \frac{X_t}{t} = v , \; \P _0 ^{(\alpha)} \text{a.s.} ,\]
and $ \forall i = 1 \dots d $ such that $ d_\alpha \cdot e_i \neq 0 $, we have 
\[ ( d _\alpha \cdot e_i )( v \cdot e_i ) > 0  ,\] 
whereas if $ d_\alpha \cdot e_i = 0 $,  
\[ \liminf _{ t \to +\infty } X_t \cdot e_i = - \infty , \; \limsup _{ t \to +\infty } X_t \cdot e_i = + \infty , \; \P _0 ^{(\alpha)} a.s.. \]
\end{itemize}

\end{theo}

As $ (X_t) _{t \in \R _+} $ and $ (Z_n) _{n \in \N} $ go through exactly the same vertices in the same order, and as the two processes stay a finite time on each vertex without exploding, recurrence and transience for the original walk $Z_n \cdot e_i$ follow from those of $X_t \cdot e_i$.

\begin{corollary}
\label{th1ralenti}
Let $d \geq 3$, for $ i = 1, \dots ,d $

\begin{itemize}
\item[i)] If $ d_\alpha \cdot e_i = 0 $, 
\[ \liminf _{ n \to +\infty } Z_n \cdot e_i = - \infty  , \;  \limsup _{ n \to +\infty } Z_n \cdot e_i = + \infty , \; \P _0 ^{(\alpha)} \text{a.s.}. \]
\item[ii)] If $ d_\alpha \cdot e_i > 0 $,
\[ \lim _{ n \to +\infty } Z_n \cdot e_i =  + \infty  , \; \P _0 ^{(\alpha)} \text{a.s.}. \]
\item[iii)] If $ d_\alpha \cdot e_i < 0 $,
\[ \lim _{ n \to +\infty } Z_n \cdot e_i =  - \infty  , \; \P _0 ^{(\alpha)} \text{a.s.}. \]
\end{itemize}

\end{corollary}

The proof of theorem \ref{thppal} allows besides to solve the problem of Kalikow's $0-1$ law in the Dirichlet case.

\begin{corollary}[Kalikow's $0-1$ law in the Dirichlet case]
\label{Kalikow}
Let $ \P ^{(\alpha)} $ be the law of the Dirichlet environment on $ \Z ^d $, $ d \geq 1 $, for the weights $ (\alpha _1, \dots , \alpha_{2d}) $, and $Z_n$ the associated random walk in Dirichlet environment. Then for all $ l \in \R ^d \setminus \{ 0 \} $, we are in one of the following cases : 
\begin{itemize}
\item 
$ \Liminf _{ n \to +\infty } Z_n \cdot l = - \infty , \; \Limsup _{ n \to +\infty } Z_n \cdot l = + \infty , \; \P _0 ^{(\alpha)} a.s. $,
\item 
$  \Lim _{ n \to +\infty } Z_n \cdot l = - \infty , \; \P _0 ^{(\alpha)} a.s. $,
\item
$ \Lim _{ n \to +\infty } Z_n \cdot l = + \infty , \; \P _0 ^{(\alpha)} a.s. $.
\end{itemize}
\end{corollary}

\begin{rmq}

Theorem \ref{thppal} also gives the existence of a deterministic asymptotic direction $\P _0 ^{(\alpha)}$ a.s. when $ d \geq 3 $ and $ d _\alpha \neq 0 $. As I was finishing this article, Tournier informed me about the existence of a more general version of theorem 1 of \cite{ST}. Using  this result instead of \cite{ST} in the proof of theorem \ref{thppal} allows to show that the asymptotic direction is $ \frac{ d _\alpha }{ | d _\alpha | } $, see \cite{Tournier2} for details.

\end{rmq}

In the transient sub-ballistic case, we also obtain the polynomial order of the magnitude of the walk's displacement :

\begin{theo}
\label{theoremelimite}
Let $d \geq 3$, $ \P ^{(\alpha)} $ be the law of the Dirichlet environment with parameters $ (\alpha _1, \dots , \alpha_{2d}) $ on $ \Z ^d $, and $Z_n$ the associated random walk in Dirichlet environment. We suppose that $ \kappa = 2 \( \sum_{i=1} ^{2d} \alpha _i \) - \max _{i=1,\dots, d} (\alpha _i + \alpha _{i+d}) \leq 1 $. Let $ l \in \{ e_1,  \dots , e_{2d} \} $ be such that $ d _\alpha \cdot l \neq 0 $. Then 
\[  \lim _{n \to + \infty} \frac{ \log( Z _n \cdot l ) }{ \log (n) } = \kappa \text{ in } \P ^{(\alpha)} \text{-probability}.\]

\end{theo}

\begin{rmq}

The directional transience shown in \cite{Tournier2} should also enable to extend the results of theorem \ref{thppal}, corollary \ref{th1ralenti} and theorem \ref{theoremelimite} from $ (e_i) _{ i = 1, \dots, 2d } $ to any $ l \in \R ^d $.

\end{rmq}

\section{Proof of theorem \ref{th1}}

We first give some definitions and notations. Let $(G,V)$ be an oriented graph. For $e \in E$, we note $ \underline{e} $ the tail of the edge, and $ \overline{e} $ his head, such that $ e = ( \underline{e} , \overline{e} ) $. The divergence operator is : div $ : \R ^E \to \R ^V $ such that : $ \forall x \in \Z^d $, 
\[ \text{div} (\theta) (x)= \sum _{e \in E, \; \underline{e} = x} \theta (e) - \sum_{e \in E, \; \overline{e} = x} \theta (e) .\]

For $ N \in \N ^* $, we set $ T _N = (\Z / N \Z)^d $ the $d$-dimensional torus of size $N$. We note $G_N = (T_N, E_N)$ the directed graph obtained by projection of $(\Z^d, E)$ on the torus $T_N$. Let $ \Omega_N $ be the space of elliptic random environments on the torus : 
\[ \Omega_N =  \{ \omega = (\omega(x,y)) _{ x \sim y} \in ] 0 , 1] ^{E_N} \text{ such that } \forall x \in T_N, \sum _{i=1}^{2d} \omega (x , x+ e_i)= 1  \} .\]
We denote by $ \P _N ^{(\alpha)}$ the law on the environment obtained by choosing independently for each $ x \in T_N $ the exit probabilities of $x$ according to a Dirichlet law with parameters $ (\alpha_1, \dots, \alpha _{2d}) $.

For $\omega \in \Omega _N$, we note $ \pi _N ^\omega $ the unique (because of ellipticity) invariant probability measure of $ Z _n ^\omega $ on the torus in the environment $\omega$. Then $ \( \frac{ \pi _N ^\omega (x) }{ \gamma ^\omega (x) } \) _{ x \in T_N } $ is an invariant measure for $ X _t ^\omega $ on the torus in the environment $\omega$, and 
\[ \tilde{\pi} _N  ^\omega (y) :=  \frac{ \frac{  \pi _N ^\omega (y) }{ \gamma ^\omega (y)} }{ \sum _{x \in T_N}  \frac{  \pi _N ^\omega (x) }{ \gamma ^\omega (x)} } \] 
is the associated invariant probability. Define 
\[ f _N (\omega) :=  N^d  \tilde{\pi}_N ^\omega (0)   \text{ and }  \Q _N ^{(\alpha)} := f_N \P _N ^{(\alpha)} ,\]
then, thanks to translation invariance, $\Q _N ^{(\alpha)}$ is an invariant probability measure on $\Omega _N$. We can now reduce theorem \ref{th1} to the following lemma.

\begin{lemma}

$ \forall p \in [ 1, \kappa ^\Lambda [ $, 
\[  \sup _{N \in \N}  \parallel  f _N  \parallel _{ L_p ( \P _N ^{(\alpha)} ) }  < +\infty .\]
\end{lemma}

Once this lemma is proved, the proof of theorem \ref{th1} follows easily, we refer to \cite{S}, pages $5,6$, where the situation is exactly the same, or to \cite{Sz}, pages $11$ and $18,19$.

\begin{proof}[Proof of lemma $1$]
This proof is divided in two main steps. First we introduce the "time-reversed environment" and prepare the application of the "time reversal invariance" (lemma $1$ of \cite{S2}, or proposition $1$ of \cite{ST}). Then we apply this invariance, and use a lemma of the type "max-flow min-cut problem".

\subsection*{Step $1$ :}
Let $ (\omega (x,y)) _{x \sim y} $ be in $ \Omega _N $. The time-reversed environment is defined by : $ \forall (x,y) \in T_N ^2$, $x \sim y$, 
\[ \check{\omega} (x,y) = \omega (y,x) \frac{ \pi _N ^\omega (y) }{ \pi _N ^\omega (x) }  .\]
We know that : $\forall x \in T_N$,
\[ \sum _{ \underline{e} = x } \alpha (e) = \sum _{ \overline{e} = x } \alpha (e) = \sum _{j=1} ^{2d} \alpha _j,\]
then $ \text{div}(\alpha)(x) = 0 $. We can therefore apply lemma $1$ of \cite{S2} which gives : if $ ( \omega (x,y) ) $ is distributed according to $ \P _N ^{(\alpha) }$, then $ (\check{\omega} (x,y)) $ is distributed according to $ \P _N ^{(\check{\alpha})} $, where $ \forall (x, y) \in E_N ^2$,
\[ \check{\alpha} (x,y) = \alpha (y,x) .\]

Let $p$ be a real, $1 < p < \kappa ^\Lambda $. We have : 
\[ (f _N (\omega))^p = \( N^d  \tilde{\pi}_N ^\omega (0) \) ^p .\]
Introducing the immediate fact that
\[ 1 =  \Sum _{x \in T_N} \tilde{\pi} _N ^\omega (x)   ,\]
it gives :
 \begin{equation*}
   (f _N (\omega))^p 
	= \( \frac{ \tilde{\pi}_N ^\omega (0) \times N^d  }{    \Sum _{x \in T_N}  \tilde{\pi} _N ^\omega (x)  } \) ^p    
 \end{equation*}
we can then use the arithmetico-geometric inequality :
 \begin{align}
   (f _N (\omega))^p 
 	&\leq \prod _{ x \in T_N }  \( \frac{ \tilde{\pi}_N ^\omega (0) }{ \tilde{\pi}_N ^\omega (x) }  \) ^{ \frac{p}{N^d} } \nonumber &\\
    &= \prod _{ x \in T_N } \(  \( \frac{ \pi _N ^\omega (0) }{ \pi _N ^\omega (x) } \) ^{ \frac{p}{N^d} }  \(  \frac{ \gamma ^\omega (x) }{ \gamma ^\omega (0) }  \) ^{ \frac{p}{N^d} }   \)   \label{fnp} &
 \end{align}
 
Take $ \theta _N : E_N \to \R _+ $, and define $ \check{\theta} _N $ by : $ \forall x \sim y$, $ \check{\theta} _N (x,y) = \theta _N (y,x) $. It is clear that 
\begin{equation}
\label{check}
 \frac{ \check{\omega} ^{ \check{\theta} _N } }{ \omega ^{ \theta _N  } } = \pi _N ^{ \text{div} ( \theta _N ) }  
\end{equation}
where by $ \lambda ^\beta $ we mean $ \prod _{e \in E_N} \lambda (e) ^{\beta (e)} $ (resp. $ \prod _{x \in T_N} \lambda (x) ^{\beta (x)} $) for any couple of functions $ \lambda , \beta $ on $ E_N $ (resp. $ T_N $). Therefore, if we choose $ \theta _N : E_N \to \R _+$ such that 
\begin{equation}
\label{div}
 \text{div} (\theta _N) = \frac{p}{N^d} \sum _{x \in T_N} (\delta _0 - \delta _x) ,
\end{equation}
(\ref{fnp}) and (\ref{check}) give us
\[ f _N ^p \leq \frac{ \check{\omega} ^{ \check{\theta} _N } }{ \omega ^{\theta_N} } \prod _{ x \in T_N } \(  \frac{ \gamma ^\omega (x) }{ \gamma ^\omega (0) }  \) ^{ \frac{p}{N^d} } .\]
We therefore only have to show that we can find $ (\theta _N) _{N \in \N} $ such that for all $N$, $ \theta _N $ satisfies (\ref{div}) and such that :
\begin{equation}
\label{finite}
 \sup _{N \in \N} \E ^{(\alpha)} \left[  \frac{ \check{\omega} ^{ \check{\theta} _N } }{ \omega ^{\theta_N} } \prod _{ x \in T_N } \(  \frac{ \gamma ^\omega (x) }{ \gamma ^\omega (0) }  \) ^{ \frac{p}{N^d} } \right] < \infty .
\end{equation}
 
\subsection*{Step $2$ :}
 Take $p>1$. We first construct a sequence $ (\theta _N) _{N \in \N} $ that satisfies (\ref{div}), and then we show that it satisfies (\ref{finite}).

\textbf{Construction of $ (\theta _N) _{N \in \N} $.}
We want to use lemma 2 of \cite{S}, which is a result of type maw-flow min-cut (see for example \cite{LP}, section $3.1$, for a general description of the max-flow min-cut problem). We first recall some definitions and notions on the matter. In an infinite graph $G = (V,E)$, a cut-set between $ x \in V $ and $ \infty $ is a subset $S$ of $E$ such that any infinite simple directed path (i.e. an infinite directed path that does not go twice through the same vertex) starting from $x$ must necessarily go through one edge in $S$. A cut-set which is minimal for inclusion is necessarily of the form :
\begin{equation}
\label{cutset}
S = \partial _+ (A) = \{ e \in E , \; \underline{e} \in A, \; \overline{e} \in A^c \}
\end{equation}
where $A$ is a finite subset of $V$ containing $x$ and such that any $y \in A$ can be reached by a directed path in $A$ starting from $x$. Let $( c (e) )_{e \in E} $ be a family of non-negative reals, called the capacities. The minimal cut-set sum between $0$ and $\infty$ is defined by : 
\[ m( (c(e))_{e \in E}  ) = \inf \{ c(S), \; S \text{ a cut-set separating } 0 \text{ and } \infty \} \]
where $ c(S) = \sum _{e \in S} c(e) $. Remark that the infimum can be taken only on minimal cut-sets, i.e. cut-sets of the form (\ref{cutset}).

Subsequent calculations will show the need for a $\theta _N$ depending on an arbitrary path $\sigma$ from $0$ to $ \Lambda ^c $. Set $N \in \N$, we define : 
\begin{align*}
  \alpha ^{(\sigma)} (e)=
  \begin{cases}
    \alpha (e) + \kappa ^\Lambda  & \text{if }  e \in \sigma \\
    \alpha (e) & \text{otherwise }
  \end{cases}
\end{align*}
Then $ m(( \alpha ^{(\sigma)} (e) )_{ e \in E_N } ) \geq \kappa ^\Lambda $. Indeed : 
\begin{itemize}
\item If some $e \in \sigma$ is in the min-cut, it is obvious.
\item Otherwise, as $ 0 \in \sigma $ the min-cut is of the form $ S = \partial _+ (K) $ with $ \sigma \subset K $ and $K$ a finite connected set of vertices. The definition of $ \kappa ^ \Lambda $ in theorem \ref{th1} gives directly $ m(( \alpha ^{(\sigma)} (e) )_{ e \in E_N } ) \geq \kappa ^\Lambda $.
\end{itemize}
Then lemma 2 of \cite{S}, with $ c(e) = \frac{p}{\kappa ^\Lambda} \alpha ^{ (\sigma) } (e) $, gives that for all $ N \geq N_0 $ there is a function $ \theta _N $ satisfying (\ref{div}) and such that $ \theta _N (e) \leq  \frac{p}{\kappa ^\Lambda} \alpha ^{(\sigma)} (e) $. 

\textbf{Preliminary computations about (\ref{finite}).}
Let $q$ and $r$ be positive reals such that $ \frac{1}{r} + \frac{1}{q} = 1 $ and $pq < \kappa ^\Lambda $. Using in a first time Hölder's inequality and then the time-reversed environment (lemma $1$ of \cite{S2}), we obtain :
 \begin{align}
	 \E ^{ (\alpha) } \left[  \frac{ \check{\omega} ^{ \check{\theta} _N } }{ \omega ^{\theta _N} } \prod _{ x \in T_N } \(  \frac{ \gamma ^\omega (x) }{ \gamma ^\omega (0) }  \) ^{ \frac{p}{N^d} }  \right]    
 	&\leq   \E ^{ (\alpha) } \left[  \omega ^{- q\theta _N}  \prod _{ x \in T_N } \(  \frac{ \gamma ^\omega (x) }{ \gamma ^\omega (0) }  \) ^{ \frac{pq}{N^d} }  \right] ^{ \frac{1}{q} } \E ^{ (\alpha) } \left[  \check{\omega} ^{ r \check{\theta} _N } \right] ^{ \frac{1}{r} }  \nonumber &\\
     &=   \E ^{ (\alpha) } \left[  \omega ^{- q\theta_N}  \prod _{ x \in T_N } \(  \frac{ \gamma ^\omega (x) }{ \gamma ^\omega (0) }  \) ^{ \frac{pq}{N^d} }  \right] ^{ \frac{1}{q} } \E ^{ (\check{\alpha}) } \left[ \omega ^{ r \check{\theta} _N } \right] ^{ \frac{1}{r} } \label{acalculer}&
 \end{align}
Define $ \alpha (x) = \sum _{ \underline{e} = x } \alpha (e)$, $\theta _N (x) = \sum _{ \underline{e} = x } \theta _N (e)$. For all $x \in T_N$, we have $ \alpha (x) = \check{\alpha} (x) = \sum _{i=0} ^{2d} \alpha _i $, we thus note $ \alpha _0 = \sum _{i=0} ^{2d} \alpha _i $. In order to simplify notations, we note $ d\lambda _\Omega = \prod _{ e \in \tilde{E}_N } d\omega(e)$, where we obtain $ \tilde{E} _N $ from $E_N$  by removing for each $x$ one arbitrary edge leaving $x$. We can now compute the first expectation in (\ref{acalculer}) : 
 
 \begin{align*}
 	&\E ^{ (\alpha) } \left[  \omega ^{- q\theta_N}  \prod _{ x \in T_N } \(  \frac{ \gamma ^\omega (x) }{ \gamma ^\omega (0) }  \) ^{ \frac{pq}{N^d} }  \right] &\\
	&=  \int \( \prod _{e \in E_N} \omega (e) ^{ \alpha (e) - 1 - q \theta _N (e) } \) \Prod _{ x \in T_N } \(  \frac{ \gamma ^\omega (x) }{ \gamma ^\omega (0) }  \) ^{ \frac{pq}{N^d} } \frac{ \Prod _{ x \in T_N } \Gamma (\alpha _x) }{ \Prod _{ e \in E_N } \Gamma (\alpha _e) }  d\lambda _\Omega &\\
 	& = \int \( \prod _{e \in E_N} \omega (e) ^{ \alpha (e) - 1 - q \theta _N (e) } \) \frac{ \( \Sum _{ \sigma : 0 \to \Lambda^c } \omega _\sigma   \) ^{pq - \frac{pq}{N^d} }  \Prod _{ x \in T_N } \Gamma (\alpha _x)  }{ \Prod _{ \substack{x \in T_N \\ x \neq 0 }} \( \Sum _{ \sigma : x \to (x+\Lambda)^c } \omega _\sigma \) ^{ \frac{pq}{N^d} } \Prod _{ e \in E_N } \Gamma (\alpha _e) }  d\lambda _\Omega  &  
 \end{align*}
where all the sums on $\sigma$ correspond to the sums on simple paths. As 
\[ \( \sum _{ \sigma : 0 \to \Lambda^c } \omega _\sigma   \) ^{pq - \frac{pq}{N^d} } \leq C  \sum _{ \sigma : 0 \to \Lambda^c } \omega _\sigma ^{pq - \frac{pq}{N^d}} ,\]
with $ C = \( \# \{ \sigma : 0 \to \Lambda ^c \} \) ^{pq - \frac{pq}{N^d}} $, we have :

 \begin{align*}
 	&\E ^{ (\alpha) } \left[  \omega ^{- q\theta_N}  \prod _{ x \in T_N } \(  \frac{ \gamma ^\omega (x) }{ \gamma ^\omega (0) }  \) ^{ \frac{pq}{N^d} }  \right] &\\
	&\leq C \sum _{ \sigma : 0 \to \Lambda^c } \int \( \prod _{e \in E_N} \omega (e) ^{ \alpha (e) - 1 - q \theta _N (e) } \)  \frac{  \omega _\sigma   ^{pq - \frac{pq}{N^d} } \Prod _{ x \in T_N } \Gamma (\alpha _x) }{ \Prod _{ \substack{ x \in T_N\\ x \neq 0 }} \( \Sum _{ \sigma : x \to (x+\Lambda)^c } \omega _\sigma \) ^{ \frac{pq}{N^d} }  \Prod _{ e \in E_N } \Gamma (\alpha _e) }   d\lambda _\Omega  & \\
	&\leq C \sum _{ \sigma : 0 \to \Lambda^c } \int  \frac{  \( \Prod _{e \in E_N} \omega (e) ^{ \alpha (e) - 1 - q \theta _N (e) } \) \omega _\sigma   ^{pq - \frac{pq}{N^d} }  \( \Prod _{ x \in T_N } \Gamma (\alpha _x) \) }{ \( \Prod _{ \substack{x \in T_N  , \; x \neq 0 }}  \omega _{\sigma_x}  ^{ \frac{pq}{N^d} }  \)  \( \Prod _{ e \in E_N } \Gamma (\alpha _e) \) }  d\lambda _\Omega  & \\
 \end{align*}
where $\sigma_x$ is an arbitrarily chosen simple path in the preceding sum. Then
\[ \E ^{ (\alpha) } \left[  \omega ^{- q\theta_N}  \prod _{ x \in T_N } \(  \frac{ \gamma ^\omega (x) }{ \gamma ^\omega (0) }  \) ^{ \frac{pq}{N^d} }  \right]   \leq   C \sum _{ \sigma : 0 \to \Lambda^c } \frac{ \Prod _{ x \in T_N } \Gamma (\alpha _0) \Prod _{e \in E_N} \Gamma \( \beta ^\sigma (e) - q \theta _N (e)  \)  }{ \Prod _{ e \in E_N } \Gamma (\alpha _e) \Prod _{ x \in T_N } \Gamma \(  \beta ^{\sigma} (x) - q \theta _N (x)  \) }  \]
with 
\[ \beta ^\sigma (x) = \sum _{ x = \underline{e} } \beta ^\sigma (e) \]
and
\[ \beta ^\sigma (e) = \alpha (e) + pq \( 1 - \frac{1}{N^d}  \) \1 _{e \in \sigma} - \sum _{ x \in T _N, x \neq 0} \frac{pq}{N^d} \1 _{e \in \sigma _x } .\]
As $\Lambda$ is finite, an edge can be in only a finite number of $\sigma _x$. We have then for all $e$, $ \sum _{ x \in T _N, x \neq 0} \frac{pq}{N^d} \1 _{e \in \sigma _x } < + \infty $. This proves that $ \beta ^\sigma $ is well defined and takes only finite values.

The second expectation in (\ref{acalculer}) is easy to compute :
\[  \E^{( \check{\alpha) }} (\omega ^{ r \check{\theta} _N } ) =  \frac{ \Prod _{e \in E_N} \Gamma ( \alpha (e) + r \theta _N (e)) \Prod _{ x \in T_N } \Gamma ( \alpha _0 ) }{ \Prod _{ x \in T_N } \Gamma ( \alpha _0 + r \check{\theta} _N (x)  )  \Prod _{e \in E_N} \Gamma ( \alpha (e) ) } .\]
We did not check that the previous expressions are well defined : we need to prove that for the given $\theta _N$, the arguments of the Gamma functions are positive. As it is a bit tedious, we delay this checking to the next point in the proof.

We now have that $ \E ^{ (\alpha) } \left[  \frac{ \check{\omega} ^{ \check{\theta} _N } }{ \omega ^{\theta _N} } \prod _{ x \in T_N } \(  \frac{ \gamma ^\omega (x) }{ \gamma ^\omega (0) }  \) ^{ \frac{p}{N^d} }  \right]  $ is smaller than :
\[ \(  C \sum _{ \sigma : 0 \to \Lambda^c } \frac{ \Prod _{x} \Gamma (\alpha _0) \Prod _{e} \Gamma \( \beta ^\sigma (e) - q \theta _N (e)  \)  }{ \Prod _{e} \Gamma (\alpha _e) \Prod _{x} \Gamma \(  \beta ^{\sigma} (x) - q \theta _N (x) \) } \) ^{ \frac{1}{q} }   \(  \frac{ \Prod _{e} \Gamma ( \alpha (e) + r \theta _N (e)) \Prod _{x} \Gamma ( \alpha _0 ) }{ \Prod _{x} \Gamma ( \alpha _0 + r \check{\theta} _N (x)  )  \Prod _{e} \Gamma ( \alpha (e) ) } \) ^{ \frac{1}{r} }  \]
which is smaller than
\[  C'  \sum _{ \sigma : 0 \to \Lambda^c } \(  \frac{ \Prod _{x} \Gamma (\alpha _0) \Prod _{e} \Gamma \( \beta ^\sigma (e) - q \theta _N (e)  \)  }{ \Prod _{e} \Gamma (\alpha _e) \Prod _{x} \Gamma \(  \beta ^{\sigma} (x) - q \theta _N (x) \) } \) ^{ \frac{1}{q} }  \(  \frac{ \Prod _{e} \Gamma ( \alpha (e) + r \theta _N (e)) \Prod _{x} \Gamma ( \alpha _0 ) }{ \Prod _{x} \Gamma ( \alpha _0 + r \check{\theta} _N (x)  )  \Prod _{e} \Gamma ( \alpha (e) ) } \) ^{ \frac{1}{r} }  \]
where $ C' = C  \( \# \{ \sigma : 0 \to \Lambda ^c \} \) ^{ \frac{1}{q} } $. We want to prove the finiteness of this expression. As we sum on a finite number of paths, we only have to show that the general term of the sum stays finite. We are reduced to prove that $ \forall \sigma : 0 \to \Lambda ^c $ simple path,
\begin{equation}
\label{acalculer2}
 \sup _{ N  \in \N }  \(  \frac{ \Prod _{ x \in T_N }  \Gamma (\alpha _0)  }{ \Prod _{ e \in E_N } \Gamma (\alpha _e) }  \(  \frac{ \Prod _{e \in E_N} \Gamma \( \beta ^\sigma (e) - q \theta _N (e)  \)  }{  \Prod _{ x \in T_N } \Gamma \(  \beta ^{\sigma} (x) - q \theta _N (x)  \) } \) ^{ \frac{1}{q} }  \(  \frac{ \Prod _{e \in E_N} \Gamma ( \alpha (e) + r \theta _N (e))  }{ \Prod _{ x \in T_N } \Gamma ( \alpha _0 + r \check{\theta} _N (x)  )  } \) ^{ \frac{1}{r} } \) < + \infty
 \end{equation}
 
\textbf{Checking that the previous Gamma functions were well defined.}
As for all $ e \in E_N$ $ \alpha (e) > 0 $ and $ \theta _N (e) \geq 0 $, the result is straightforward except for $  \Gamma \( \beta ^\sigma (e) - q \theta _N (e)  \)$ and $ \Gamma \(  \beta ^{\sigma} (x) - q \theta _N (x)  \) $. By construction of $\theta _N$, we know that $ \beta ^{\sigma} (e) - q \theta _N (e)  \geq \beta ^{\sigma} (e) - \frac{pq}{\kappa ^\Lambda} \alpha ^{(\sigma)} (e) $. Then we just have to check the positivity of this second expression. Take $ e \in E_N $ :

\begin{itemize}

\item If $ e \in \sigma $, then
\begin{align*}
	  \beta ^{\sigma} (e) - \frac{pq}{\kappa ^\Lambda} \alpha ^{(\sigma)} (e) 
	 &= \alpha (e) - \frac{pq}{\kappa ^\Lambda} (\alpha (e) + \kappa ^\Lambda ) + pq - \frac{pq}{N^d} - \sum _{ \substack{ x \in T_N \\ x \neq 0 }} \frac{pq}{N^d} \1 _{ e \in \sigma _x }  &\\
	 &=  \alpha (e)\( 1 - \frac{pq}{\kappa ^\Lambda}  \) - \frac{pq}{N^d} \Big( 1 + \sum _{ \substack{x \in T_N \\ x \neq 0} } \1 _{ e \in \sigma _x } \Big) &
\end{align*}
As we assumed $ pq < \kappa ^\Lambda $ and $ \kappa ^\Lambda > 1 $, $ \alpha (e) \( 1 - \frac{pq}{\kappa ^\Lambda} \) > 0 $. The second term can be made as small as needed by choosing $N$ big enough. Then $ \beta ^{\sigma} (e) - \frac{pq}{\kappa ^\Lambda} \alpha ^{(\sigma)} (e) > 0 $ for $N$ big enough.

\item If $ e \notin \sigma $, then
\begin{align*}
	\beta ^{\sigma} (e) - \frac{pq}{\kappa ^\Lambda} \alpha ^{(\sigma)} (e) 
	&=  \alpha (e) - \frac{pq}{\kappa ^\Lambda} \alpha (e) - \sum _{ \substack{ x \in T_N \\ x \neq 0 }} \frac{pq}{N^d} \1 _{ e \in \sigma _x }  &\\
	&\geq \alpha (e)\( 1 - \frac{pq}{\kappa ^\Lambda}  \) - (\sharp \Lambda) \frac{pq}{N^d}	&
\end{align*}
As before, by choosing $N$ big enough we make sure that it is positive. Remark that for N big enough, $ \min _{i = 1 \dots 2d} \alpha_i \( 1 - \frac{pq}{\kappa ^\Lambda}  \) - (\sharp \Lambda) \frac{pq}{N^d} $ is also positive, and it is a uniform lower bound of $ \beta ^{\sigma} (e) - q \theta _N (e) $, for all $ e \notin \sigma $.
\end{itemize}

\textbf{Proof of (\ref{acalculer2}).}
As $\sigma$ is a finite path, the above tells us that there exists $ \epsilon > 0 $ such that :
\[ \forall e \in \sigma , \; \epsilon \leq \beta ^{\sigma} (e) - \frac{pq}{\kappa ^\Lambda} \alpha ^{(\sigma)} (e)  = \alpha (e)\( 1 - \frac{pq}{\kappa ^\Lambda}  \) - \frac{pq}{N^d} \Big( 1 + \sum _{ \substack{x \in T_N \\ x \neq 0} } \1 _{ e \in \sigma _x } \Big)  \leq \alpha (e) , \] 
and the same is true for $ \alpha (x) $ by summing on $e$. Define : 
\[ A_1 ^\sigma =  \( \frac{ \Prod _{x \in e \in \sigma } \Gamma (\alpha _0 ) }{ \Prod _{ e \in \sigma } \Gamma ( \alpha (e) ) } \frac{ \Prod _{ e \in \sigma } \sup _{ [ \epsilon , \max _i \alpha_i ] } \Gamma (s) }{ \Prod _{ x \in e \in \sigma } \inf _{ [ \epsilon , \max _i \alpha_i ]} \Gamma(s)  }  \) ^\frac{1}{q} \]
\[ A_2 ^\sigma =  \( \frac{ \Prod _{x \in e \in \sigma } \Gamma (\alpha _0 ) }{ \Prod _{ e \in \sigma } \Gamma ( \alpha (e) ) } \frac{ \Prod _{ e \in \sigma } \sup _{ [ \alpha (e) , \alpha (e) (1+r) + r \kappa ^\Lambda ] } \Gamma (s) }{ \Prod _{ x \in e \in \sigma } \inf _{ [ \alpha (e) , \alpha (e) (1+r) + r \kappa ^\Lambda ] } \Gamma(s)  }  \) ^\frac{1}{r} \]
We have then, for any fixed $\sigma$ : 
\begin{align*}
	& \frac{ \Prod _{ x \in T_N }  \Gamma (\alpha _0)  }{ \Prod _{ e \in E_N } \Gamma (\alpha _e) }  \(  \frac{ \Prod _{e \in E_N} \Gamma \( \beta ^\sigma (e) - q \theta _N (e)  \)  }{  \Prod _{ x \in T_N } \Gamma \(  \beta ^{\sigma} (x) - q \theta _N (x)  \) } \) ^{ \frac{1}{q} }  \(  \frac{ \Prod _{e \in E_N} \Gamma ( \alpha (e) + r \theta _N (e))  }{ \Prod _{ x \in T_N } \Gamma ( \alpha _0 + r \check{\theta} _N (x)  )  } \) ^{ \frac{1}{r} } &\\
	&\leq A_1 ^\sigma  A_2^\sigma  \frac{ \Prod _{ x \notin \sigma }  \Gamma (\alpha _0)  }{ \Prod _{ e \notin \sigma } \Gamma (\alpha _e) }  \(  \frac{  \Prod _{e \notin \sigma} \Gamma \( \beta ^{\sigma} (e) - q \theta _N (e)  \)  }{  \Prod _{ x \notin \sigma } \Gamma \(   \beta ^{\sigma} (x) - q \theta _N (x) \) } \) ^{ \frac{1}{q} }  \(  \frac{ \Prod _{e \notin \sigma} \Gamma ( \alpha (e) + r \theta _N (e))  }{ \Prod _{ x \notin \sigma } \Gamma ( \alpha _0 + r \check{\theta} _N (x)  )  } \) ^{ \frac{1}{r} }  &\\
	&\leq A_1 ^\sigma A_2 ^\sigma \exp \(  \Sum _{ \substack{ e \in E_N \\ e \notin \sigma } }  \nu \(  \alpha (e) , \theta _N (e), \beta ^{\sigma} (e)  \) - \Sum _{ \substack { x \in T_N \\ x \notin \sigma } }  \tilde{ \nu } \(  \alpha _0 , \theta _N (x), \beta ^{ \sigma} (x)  \) \)  &
\end{align*}
with : 
\[  \nu \(  \alpha (e) , \theta _N (e), \beta ^{\sigma} (e)  \)  = \frac{1}{r} \ln \Gamma ( \alpha (e) + r \theta _N (e) ) + \frac{1}{q} \ln \Gamma ( \beta ^{ \sigma } (e)  - q \theta _N (e) )  - \ln \Gamma ( \alpha (e)  )  \]
\[ \tilde{ \nu } \(  \alpha _0 , \theta _N (x), \beta ^{ \sigma} (x)  \) =  \frac{1}{r} \ln \Gamma ( \alpha _0 + r \theta _N (x) + \frac{pr}{N^d} ) + \frac{1}{q} \ln \Gamma ( \beta ^{ \sigma } (x) - q \theta _N (x)  )  - \ln \Gamma ( \alpha _0 ) \]
(the $ \frac{pr}{N^d} $ comes from the fact that $  \forall x \neq 0, \quad  \theta (x) - \check{\theta} (x)  =  \text{div} (\theta) (x) = - \frac{p}{N^d} $ ).
We set $ \underline{ \alpha } = \min _{i \in [\![ 1, 2d ]\!] } \alpha _i $ and $ \overline{ \alpha } = \max _{i \in [\![ 1, 2d ]\!] } \alpha _i $. Then $ \forall e \in E_N, \quad \underline{\alpha} \leq \alpha (e) \leq \overline{\alpha} $, $ \forall e \notin \sigma \quad q \theta _N (e) \leq \frac{pq}{\kappa ^\Lambda} \alpha (e) $ and $ pq < \kappa ^\Lambda $. Taylor's inequality gives : $ \forall e \notin \sigma, \forall x \notin \sigma $, 
\begin{align*}
  \begin{cases}
    \abs{ \nu \(  \alpha (e) , \theta _N (e), \beta ^{\sigma} (e)  \)  }   \leq  C_1 \(  \theta _N (e) ^2 + \frac{pq}{N^d} \) &  \\
    \abs{ \tilde{ \nu } \(  \alpha _0 , \theta _N (x), \beta ^{ \sigma} (x)  \) }   \leq  C_2 \(  \theta _N (x) ^2 + \frac{p}{N^d} + \frac{pq}{N^d} \) &
  \end{cases}
\end{align*}
with $C_1$ and $C_2$ positive constants. Then we can find a constant $ C_3 > 0 $ independent of $ N \geq N_0 $ such that :
\begin{align*}
	& \frac{ \Prod _{ x \in T_N }  \Gamma (\alpha _0)  }{ \Prod _{ e \in E_N } \Gamma (\alpha _e) }  \(  \frac{ \Prod _{e \in E_N} \Gamma \( \beta ^\sigma (e) - q \theta _N (e)  \)  }{  \Prod _{ x \in T_N } \Gamma \(  \beta ^{\sigma} (x) - q \theta _N (x)  \) } \) ^{ \frac{1}{q} }  \(  \frac{ \Prod _{e \in E_N} \Gamma ( \alpha (e) + r \theta _N (e))  }{ \Prod _{ x \in T_N } \Gamma ( \alpha _0 + r \check{\theta} _N (x)  )  } \) ^{ \frac{1}{r} } &\\
	&\leq \exp \( C_3 \(  \sum _{e \in E_N} \theta _N (e) ^2 + \sum _{x \in T_N} \theta _N (x) ^2  \)  \). &
\end{align*}
According to lemma $2$ of \cite{S}, this is bounded by a finite constant independent of $N$. It follows that the supremum on $N$ is finite too. This concludes the argument for any fixed $ \sigma $ and proves (\ref{acalculer2}). This proves the lemma.

\end{proof}

\section{Proof of theorem \ref{thppal} and corollary \ref{Kalikow}}

To obtain results on the initial random walk $Z_n$, we need some estimates on our acceleration function $ \gamma ^\omega $. In particular, we will need the following lemma :
\begin{lemma}
\label{finiteexpectation}
For all $x \in \Z ^d$ and $s < \kappa$, 
\[ \E ^{(\alpha)} \(  ( \gamma ^\omega (x) ) ^s  \)  < + \infty .\]
\end{lemma}
As its proof is quite computational, we defer it to the appendix. Remark that it is nevertheless quite easy to get a weaker bound : $ \gamma ^\omega (0) = \frac{1}{\sum \omega _\sigma} \leq \frac{1}{\omega _{ \sigma_1 }} $, where the sum is on all $\sigma$ finite simple paths from $0$ to $\Lambda ^c$, and where $ \sigma_1 $ is the path from $0$ to $\Lambda ^c$ going only through edges $(n e_1, (n+1) e_1 )$. Then $ \E ^{ (\alpha) } \(   \gamma ^\omega (0) ^\lambda \) \leq \E ^{ (\alpha) } \( \frac{1}{\omega _{ \sigma_1 }} ^\lambda \) < + \infty $ for all $ \lambda < \alpha_1 $.

Theorem \ref{thppal} is based on classical results on ergodic stationary sequences, see \cite{D} pages $342-344$. We need another preliminary lemma.

\begin{lemma}
\label{stationaryergodic}
$ \Q ^{(\alpha)} $ is ergodic and equivalent to $ \P ^{(\alpha)} $. Set $ \Delta _i = X _{i} - X _{ {i-1} } $, $ i \in \N $, then $ \Delta _i $ is stationary and ergodic under $ \Q ^{(\alpha)} [ P_0 ^\omega (.) ] $.
\end{lemma}

\begin{proof}
The proof of the first point is easily adapted from chapter 2 of \cite{Sz}, by replacing the discrete process by the continuous process : we use the continuous martingales convergence theorems, and the continuous version of Birkhoff's theorem (see for example \cite{K}, pages $9 - 11$). 

For the second point, as $ \Q ^{(\alpha)} $ is an invariant probability for $ \overline{\omega}_t $, it is straightforward that $ \Delta _i $ is stationary. It remains to prove ergodicity. Set $ A \subset \( \Z ^d \) ^\N $ a measurable set such that $ \forall t $, $ \theta _t ^{-1} (A) = A$ with $ \theta _t $ the time-shift. We note 
\[ r(x , \omega) = P _x ^\omega ( ( \Delta _i \in A ) )  \text{ and }  r(\omega) = r(0 , \omega) .\] 
We have $ \forall \omega \in \Omega $, 
\begin{equation}
\label{limiteind}
 \lim _{n \to \infty} r ( X _{n} , \omega ) = \1 _A ((\Delta _i)) , \;  P_x ^\omega  \text{ a.s.} .
\end{equation}
Indeed, setting $ \mathcal{F}_n = \sigma ( (X_t)_{t \leq n} ) $ gives :
\[ P _x ^\omega \( (\Delta _i) \in A | \mathcal{F}_n \) = P _x ^\omega \( (\Delta _{i+n}) \in A | \mathcal{F}_n \) = P _{X_{n}} ^\omega \( (\Delta _i) \in A \) = r ( X _{n} , \omega ) ,\]
then $ r ( X _{n} , \omega ) $ is a (closed) bounded martingale and we have the wanted limit (\ref{limiteind}) by a.s. convergence, as $ \1 _A ((\Delta _i)) $ is $ \mathcal{F} _\infty $-measurable. Remark that $ r(X_n,\omega) = r( \overline{\omega} _n ) $. The application of Birkhoff's ergodic theorem (\cite{D}, page 337) for the time-shift of size $1$ gives 
\[ \lim _{n \to \infty} \frac{1}{n} \sum _{k=1} ^n r (X _{k} , \omega) = \E ^{\Q ^{(\alpha)}} (r(\omega)) \; , \;  P _0 ^\omega  \text{a.s.}.\]
Comparing with (\ref{limiteind}), it implies that $ \E ^{\Q ^{(\alpha)}} (r(\omega)) \in \{ 0 , 1 \} $.

\end{proof}

\begin{lemma}
\label{excursions}

Let $ D (l,n) = \max _{t \in [0,1]} | (X_{n+t} - X_n) \cdot l | $ be the maximum distance travelled by the walk in direction $l = e_1, \dots, e_{2d}$, during a time $[n,n+1]$, $n \in \N$. Choose $R_\Lambda$ such that $ \Lambda $ is included in the ball $ B ( 0, R _\Lambda ) $. Then
\[ \P ^{(\alpha)} \( D (l,n) \geq  2 k R_\Lambda  \) \leq \frac{ C ^k }{ k ! } \]
where $C$ is a positive constant depending only on the parameters $(\alpha_1, \dots, \alpha_{2d})$.

\end{lemma}

\begin{proof}
Let $N$ be the number of visits of $0$ before exiting $\Lambda$. The random variable $N$ follows a geometric law of parameter $ p_N := \frac{1}{ G^{ \omega , \Lambda }(0,0) } $ the inverse of the Green function killed at the exit time of $\Lambda$. We note $T$ the total time spent on $0$ before exiting $\Lambda$, $ T = \frac{ 1 }{ \gamma ^\omega (0)} \sum _{i=1} ^N E_i $, where the $E_i$ are independent exponential random variables of parameter $1$. Set $ \epsilon > 0$.
\[ P ^\omega \( T \leq \epsilon | N \) = P ^\omega \( \sum _{i=1} ^N E_i \leq \gamma ^\omega (0) \epsilon | N \) = e ^{ - \gamma ^\omega (0) \epsilon } \sum _{k=N} ^{+ \infty} \frac{ ( \gamma ^\omega (0) \epsilon ) ^k }{k!}  .\]
Then
\begin{align*}
	 P ^\omega \( T \leq \epsilon \) 
	& = e ^{ - \gamma ^\omega (0) \epsilon } E ^\omega \( \sum _{k=N} ^{+ \infty} \frac{ ( \gamma ^\omega (0) \epsilon ) ^k }{k!} \)  &\\
	& = e ^{ - \gamma ^\omega (0) \epsilon } \sum _{n=1} ^{+ \infty} \sum _{k=n} ^{+ \infty} \frac{ ( \gamma ^\omega (0) \epsilon ) ^k }{k!}  p_N  \( 1 - p_N  \)^{n-1}  &\\
	& = e ^{ - \gamma ^\omega (0) \epsilon } \sum _{k=1} ^{+ \infty} \sum _{n=1} ^k \frac{ ( \gamma ^\omega (0) \epsilon ) ^k }{k!}  p_N  \( 1 - p_N  \)^{n-1}  &\\
	& = e ^{ - \gamma ^\omega (0) \epsilon } \sum _{k=1} ^{+ \infty} \frac{ ( \gamma ^\omega (0) \epsilon ) ^k }{k!}  p_N  \frac{ 1 - (1- p_N) ^k }{p_N}  &\\
	& = 1 - e ^{- p_N \gamma ^\omega (0) \epsilon} = 1 - e ^{- \frac {\gamma ^\omega (0)}{G ^{\omega,\Lambda}(0,0) } \epsilon} &
\end{align*}
For all $a > 0$, let $ 0 < \lambda < \kappa $,
\begin{align*}
	& \P ^{(\alpha)} \( T \leq \epsilon \) &\\
	& = \E ^{(\alpha)} \( 1 - e ^{- \frac {\gamma ^\omega (0)}{G ^{\omega,\Lambda}(0,0) } \epsilon}   \)  &\\
	& = \E ^{(\alpha)} \(  (1 - e ^{- \frac {\gamma ^\omega (0)}{G ^{\omega,\Lambda}(0,0) } \epsilon} ) \1 _{ \{ \frac{ \gamma ^\omega (0) }{ G ^{\omega,\Lambda} (0,0) } \geq a \} }   \) + \E ^{(\alpha)} \(  (1 - e ^{- \frac {\gamma ^\omega (0)}{G ^{\omega,\Lambda}(0,0) } \epsilon} ) \1 _{ \{ \frac{ \gamma ^\omega (0) }{ G ^{\omega,\Lambda} (0,0) } < a \} }   \) &\\   
	& \leq \P ^{(\alpha)} \(  \frac{ \gamma ^\omega (0) }{ G ^{\omega,\Lambda} (0,0) } \geq a   \)  +  \E ^{(\alpha)} \(  \frac {\gamma ^\omega (0)}{G ^{\omega,\Lambda}(0,0) } \epsilon  \1 _{ \{ \frac{ \gamma ^\omega (0) }{ G ^{\omega,\Lambda} (0,0) } < a \} }   \) &\\
	& \leq \P ^{(\alpha)} \(   \gamma ^\omega (0)  \geq a   \)  +  \E ^{(\alpha)} \(  \frac {\gamma ^\omega (0)}{G ^{\omega,\Lambda}(0,0) } \epsilon  \1 _{ \{ \frac{ \gamma ^\omega (0) }{ G ^{\omega,\Lambda} (0,0) } < a \} }   \) &\\
	& \leq \frac{ \E ^{ (\alpha) } \(   \gamma ^\omega (0) ^\lambda \) }{ a ^\lambda } + a \epsilon \P ^{(\alpha)} \(  \frac{ \gamma ^\omega (0) }{ G ^{\omega,\Lambda} (0,0) } < a   \) &
\end{align*}
As $ \lambda < \kappa $, lemma \ref{finiteexpectation} gives : 
\[ \P ^{(\alpha)} \( T \leq \epsilon \) \leq \frac {C} { a ^ \lambda } + a \epsilon \]
with $C$ a positive constant independent of $a$. Then for $ a = \epsilon ^{ - \frac{1}{\lambda + 1} }$ we have : 
\[ \P ^{(\alpha)} \( T \leq \epsilon \)  \leq  ( C + 1 ) \epsilon  ^{ \frac{ \lambda }{ \lambda + 1 } } .\]
 
If $ D (l,n) \geq  2 k R_\Lambda $, the walk went through at least $k$ distinct sets $X_t+\Lambda$ of empty intersection. The time spent in such a set is bigger than the time spent on one point in the set, and those times are independent in disjoint sets (because the environments in the sets are independent). We get (for $T_1, \dots, T_k$ i.i.d. of same law as $T$) :
\begin{align*}
	& \P ^{(\alpha)} \( D (l,n) \geq  2 k R_\Lambda  \) &\\
	& = \P ^{(\alpha)} \(  \exists \epsilon_1, \dots, \epsilon_k \text{ such that } \sum _{i=1} ^k \epsilon_i \leq 1 \text{ and } T_1 \leq \epsilon_1, \dots, T_k \leq \epsilon_k  \)  &\\
	& \leq (C+1) ^k \int _{\sum \epsilon_i \leq 1}  \prod _{i=1} ^k \epsilon_i ^{ \frac{ \lambda }{ \lambda + 1 } } d\epsilon_1 \dots d\epsilon_k &\\   
	& =  (C+1) ^k  \frac{ \Gamma ( \frac{ \lambda }{ \lambda + 1 } +1 )^k }{ \Gamma ( k \frac{ \lambda }{ \lambda + 1 } + k +1 ) } &\\
	& \leq \frac{ ( (C+1) \Gamma ( \frac{ \lambda }{ \lambda + 1 } +1 ) ) ^k  }{ k ! }  &
\end{align*}
This concludes the proof of the lemma.

\end{proof}

\begin{proof}[Proof of theorem \ref{thppal}]
Lemma \ref{stationaryergodic} gives that the sequence $ ( \Delta _i ) _{i \in \N}$ is stationary and ergodic under $ \Q ^{(\alpha)} \( P_0 ^\omega (.) \) $. We apply Birkhoff's ergodic theorem to the $\Delta_i$ to get a law of large numbers : 
\[  \frac{ X _{k} }{ k } \to _{k \to \infty, \; k \in \N}  \E ^{ \Q ^{(\alpha)} } \left[  E _0 ^\omega ( X _{1} )  \right] \; , \; \Q _0 ^{(\alpha)}  \text{ a.s. and thus } \P _0 ^{(\alpha)}  \text{ a.s. }.\]

If $ d_\alpha \cdot e_i = 0 $, the symmetry of the law of the environment gives $ \E ^{ \Q ^{(\alpha)} } \left[  E _0 ^\omega ( X _{1} )  \right] \cdot e_i = 0 $. Then $\frac{X_{k}}{k} \to 0 $ when $ d_\alpha = 0 $. Furthermore theorem 6.3.2 of \cite{D} gives that the processes $ X _{k} $ is directionally recurrent when $ d_\alpha \cdot e_i = 0 $. As $X_t$ stays only a finite time on each vertex before the next jump, directional recurrence for $ (X_k) _{k \in \N} $ implies directional recurrence for $ (X_t) _{t \in \R _+} $ (the probability to come back to $0$ after a finite time is $1$).

For $ l \in \R ^d $, we note $ A _l = \{  X_{t_k} \cdot l \to \infty \} $, where $(t_k) _{k \in \N}$ are the jump times. If $ l \neq 0 $ and if $  \P _0 ^{(\alpha)} (A_l) > 0 $, Kalikow's $ 0-1 $ law (\cite{Kalikow}, \cite{ZM} proposition $3$) gives $ \P _0 ^{(\alpha)} ( A_l \cup A_{-l} ) = 1 $. Suppose that $ d_\alpha \cdot e_i > 0 $ then (\cite{ST}) $ \P _0 ^{(\alpha)}(A_{e_i}) > 0 $, this implies that $ (X _{t_k} \cdot e_i) _{k \in \N} $ visits $ 0 $ a finite number of times $ \Q _0 ^{(\alpha)} $ a.s.. Then $ (X _{k} \cdot e_i) _{k \in \N} $ visits $ 0 $ a finite number of times $ \Q _0 ^{(\alpha)} $ a.s. (as $X_t$ stays only a finite time on each vertex). Theorem 6.3.2 of \cite{D} and Birkhoff's ergodic theorem give then : $ E ^{ \Q ^{(\alpha)} } ( E ^\omega (X_{1}) ) \cdot e_i > 0 $.

We now consider the limit for the continuous-time walk. For $t>0$, we set $ k = \floor{t} $. Then for all $i = 1, \dots, 2d$,
\[  X_k \cdot e_i - D(e_i,k)  \leq    X_t \cdot e_i   \leq    X_k \cdot e_i + D(e_i,k) .\]
Then 
\[ \frac{ X_k \cdot e_i }{ k - 1 }  -  \frac{ D(e_i,k) }{ k - 1 }  \leq  \frac{ X_t \cdot e_i }{ t }   \leq   \frac{ X_k \cdot e_i }{ k } +  \frac{ D(e_i,k) }{ k }   .\]
Lemma \ref{excursions} gives : for $ \epsilon > 0$,
\[ \sum _{k=1} ^{+ \infty}  \P ^{(\alpha)} \(  \abs{ \frac {D (e_i,k)}{k} } \geq  \epsilon  \) \leq \sum _{k=1} ^{+ \infty}  \frac{ C ^{  \frac{k \epsilon}{2 R_\Lambda}  } }{ \( \frac{k \epsilon}{2 R_\Lambda} \) ! } < + \infty .\]
Then by Borel-Cantelli's lemma, $ \frac{ D(e_i,k) }{ k } \to _{t \to +\infty}  0 \; , \; \P _0 ^{(\alpha)}  \text{ a.s. }$. It gives 
\[ \lim _{t \to +\infty} \frac{ X_t }{t} = \lim _{k \to +\infty} \frac{ X_k }{k} = \E ^{ \Q ^{(\alpha)} } \left[  E _0 ^\omega ( X _{1} )  \right] \; , \; \P _0 ^{(\alpha)}  \text{ a.s. }.\] 
This gives the directional transience in the case $ d_\alpha \cdot e_i > 0 $, and finishes the proof. 

\end{proof}

\begin{proof}[Proof of corollary \ref{Kalikow}]
We prove as in the proof of theorem \ref{thppal} that 
\begin{equation}
\label{Birkhoff}
 \lim _{t \to +\infty} \frac{ X_t }{t} = \lim _{k \to +\infty} \frac{ X_k }{k} = \E ^{ \Q ^{(\alpha)} } \left[  E _0 ^\omega ( X _{1} )  \right] \; , \; \P _0 ^{(\alpha)}  \text{ a.s. }.
\end{equation}
We still note $ A _l = \{  X_{t_k} \cdot l \to \infty , \; k \in N\} $. Suppose that $ \P _0 ^{(\alpha)}(A_l) > 0 $. Then $ \P _0 ^{(\alpha)} ( A_l \cup A_{-l} ) = 1 $ (\cite{Kalikow}, \cite{ZM} proposition $3$). It allows to find a finite interval $I$ of $\R$, of positive measure, containing $0$ and such that $(X_{t_k} \cdot l) _{k \in \N}$ goes a finite number of times in $I$,  $\P _0 ^{(\alpha)}  \text{ a.s. and thus } \Q _0 ^{(\alpha)}  \text{ a.s.}$. As before, it implies that $ (X_{k} \cdot l) _{k \in \N} $ goes a finite number of times in $I$, $\Q _0 ^{(\alpha)}$ a.s.. We can then apply the theorem of \cite{Atkinson} to $(X_k)_{k \in \Z}$ (obtained via the extension of $(X_t)_{t \in \R_+}$ to $t \in \R$) to get $\E ^{ \Q ^{(\alpha)} } \left[  E _0 ^\omega ( X _{1} \cdot l )  \right] \neq 0$. We then deduce from (\ref{Birkhoff}) that : $ X_t \cdot l \to _{ t \to \infty }  + \infty \; \P _0 ^{(\alpha)}  \text{ a.s. }$ if $ \E ^{ \Q ^{(\alpha)} } \left[  E _0 ^\omega ( X _{1} \cdot l )  \right] > 0 $, $ X_t \cdot l \to _{ t \to \infty }  - \infty \; \P _0 ^{(\alpha)}  \text{ a.s. }$ else-wise.

As $ (X_t) _{t \in \R _+} $ and $ (Z_n) _{n \in \N} $ go through exactly the same vertexes in the same order, and as the two processes stay a finite time on each vertex, without exploding (see lemma \ref{excursions}), recurrence and transience for $Z_n \cdot l$ follows from those of $X_t \cdot l$. This gives as a consequence Kalikow's $0-1$ law in the $ d \geq 3 $ Dirichlet case. 

The $0-1$ law is true in the general case of random walks in random environments for $d=1$ and $d=2$ (see respectively Solomon (\cite{Solomon}) and Zerner and Merkl (\cite{ZM})), it concludes the proof.

\end{proof}

\section{Proof of theorem \ref{theoremelimite}}

To prove the result, we need a preliminary theorem on the polynomial order of the hitting times of the walk.

\begin{theo}
\label{tpsatteinte}
Let $d \geq 3$, $ \P ^{(\alpha)} $ be the law of the Dirichlet environment with parameters $ (\alpha _1, \dots , \alpha_{2d}) $ on $ \Z ^d $, and $Z_n$ the associated random walk in Dirichlet environment. We suppose that $ \kappa = 2 \( \sum_{i=1} ^{2d} \alpha _i \) - \max _{i=1,\dots, d} (\alpha _i + \alpha _{i+d}) \leq 1 $. Let $ l \in \{ e_1,  \dots , e_{2d} \} $ be such that $ d _\alpha \cdot l \neq 0 $. Let $ T _n ^{l,Z} = \inf_i \{ i \in \N | Z_i \cdot l \geq n  \} $ be the hitting time of the level $n$ in direction $l$, for the non-accelerated walk $Z$. Then :
\[  \lim _{n \to + \infty} \frac{ \log( T _n ^{l,Z} ) }{ \log (n) } = \frac{1}{\kappa} \text{ in } \P ^{(\alpha)} \text{-probability}.\]

\end{theo} 
 
\begin{proof}

\textbf{Upper bound}

Define $A(t) = \int _0 ^t \gamma ^\omega (X_s) ds $. Then $ X _{A ^{-1}(t)} $ is the continuous-time Markov chain whose jump rate from $x$ to $y$ is $ \omega (x,y) $. This Markov chain has asymptotically the same behaviour as $Z_n$, then we only have to prove that $ \lim _{n \to + \infty }  \frac{ \log( A( T _n ^{l,X} ) ) }{ \log (n) }  \leq \frac{1}{\kappa} $, with $ T _n ^{l,X} = \inf_t \{ t \in \R _+ | X_t \cdot l \geq n  \} $.

Set $ 0 < \alpha < \kappa $, and take $ \beta $ such that $\alpha < \beta < \kappa$. Using first Markov's inequality and then the inequality $ (\sum _{i=1} ^j \lambda _i) ^\epsilon \leq \sum _{i=1} ^j (\lambda _i) ^\epsilon$ for $ \epsilon < 1$ gives :
 \begin{align*}
  \P ^{ (\alpha) } \(  \frac{ A(t) }{ t ^{\frac{1}{\alpha}} }  \geq  x  \)
	&\leq  \frac{1}{ x^\beta t^{\frac{\beta}{\alpha}} }  \E  \(    \( \int _0 ^t \gamma ^\omega (X_s) ds \)   ^\beta   \)    &\\
	&\leq  \frac{1}{ x^\beta t^{\frac{\beta}{\alpha}} }  \E  \(  \sum _{i=1} ^{ \ceil{t} }  \(  \int _{i-1} ^i \gamma ^\omega (X_s) ds  \) ^\beta   \)    &\\
 &=  \frac{1}{ x^\beta t^{\frac{\beta}{\alpha}} }  \sum _{i=1} ^{ \ceil{t} }  \E  \(   \(  \int _{i-1} ^i \gamma ^\omega (X_s) ds  \) ^\beta   \)    &
 \end{align*}
where $\ceil{t}$ represents the upper integer part of $t$. Let $D_i = \max _{ l \in \{ e_1, \dots, e_{2d} \} } ( D (l,i) )$ (cf. lemma \ref{excursions} for the definition of $D (l,i)$). Splitting the expectation depending on the value of $D_i$ gives :
 \begin{align*}
  \E  \(   \(  \int _{i-1} ^i \gamma ^\omega (X_s) ds  \) ^\beta   \)
	&=  \E \(  \sum _{k=0} ^{+\infty} \1 _{ \{ D_i = k \} }  \(  \int _{i-1} ^i \gamma ^\omega (X_s) ds  \) ^\beta   \)  &\\
	&\leq   \sum _{k=0} ^{+\infty}  \E \(  \1 _{ \{ D_i = k \} }  \(   \sum _{x \in B(X_i,k)}  \gamma ^\omega (x)  \) ^\beta   \)  &\\
 &\leq   \sum _{k=0} ^{+\infty}  \E \(  \1 _{ \{ D_i = k \} } ^p \) ^{\frac{1}{p}}  \E \(  \(   \sum _{x \in B(X_i,k)}  \gamma ^\omega (x)  \) ^{q\beta}   \) ^{\frac{1}{q}}  &
 \end{align*}
where $ B(X_i,k) = \{  x \in \Z ^d | \max _{j=1, \dots, 2d} | (X_i - x) \cdot e_j | \leq k \}$, and $ \frac{1}{p} + \frac{1}{q} = 1 $. For the last Hölder's inequality, we chose $q>1$ such that $ q \beta < \kappa $.

In the following, $c$ and $C$ will be finite constants, that can change from line to line. As $ \P ( D_i = k ) \leq \P ( D_i \geq k ) \leq  c \frac{C ^k}{k !} $ by lemma \ref{excursions} and $ q \beta < 1 $ we get :

 \begin{align*}
  \E  \(   \(  \int _{i-1} ^i \gamma ^\omega (X_s) ds  \) ^\beta   \)
	&\leq   \sum _{k=0} ^{+\infty} c \frac{C ^k}{k !} \E \(   \(   \sum _{x \in B(X_i,k)}  \gamma ^\omega (x)  \) ^{q\beta}   \) ^{\frac{1}{q}}  &\\
    &\leq   \sum _{k=0} ^{+\infty} c \frac{C ^k}{k !} \sum _{x \in B(X_i,k)}  \E \(  \gamma ^\omega (x)  ^{q\beta}   \) ^{\frac{1}{q}} &
 \end{align*}

As the Dirichlet laws are iid, the value of the expectation is independent of $x$. Lemma \ref{finiteexpectation} then gives a uniform finite bound for all $x$.
 \begin{align*}
  \P ^{ (\alpha) } \(  \frac{ A(t) }{ t ^{\frac{1}{\alpha}} }  \geq  x  \)
 	&\leq  \frac{1}{ x^\beta t^{\frac{\beta}{\alpha}} }  \sum _{i=1} ^{ \ceil{t} } \sum _{k=0} ^{+\infty} c \frac{C ^k}{k !} \sum _{x \in B(X_i,k)} 1 &\\
	&=  \frac{1}{ x^\beta t^{\frac{\beta}{\alpha}} }  \sum _{i=1} ^{ \ceil{t} } \sum _{k=0} ^{+\infty} c \frac{C ^k}{k !} (2k+1)^d  &\\
    &\leq  \frac{ \ceil{t} }{ x^\beta t^{\frac{\beta}{\alpha}} }  \sum _{k=0} ^{+\infty} c  k^d \frac{C ^k}{k !}    &\\
 &\leq  c \frac{ t^{1 - \frac{\beta}{\alpha}}  }{ x^\beta }    &
 \end{align*}

As $ \beta > \alpha $, it implies that $ \frac{A(t)}{t ^{ \frac{1}{\alpha} } } \to _{t \to +\infty} 0 $ in $\P ^{(\alpha)}$-probability, for all $ \alpha < \kappa $. Then, $ \frac{A(T _n ^{l,X})}{(T _n ^{l,X}) ^{ \frac{1}{\alpha} } } \to _{t \to +\infty} 0 $ in $\P ^{(\alpha)}$-probability. Theorem \ref{thppal} gives 
\[ \lim _{ t \to +\infty } \frac{X_t \cdot l}{t} = v \cdot l \neq 0 , \; \P _0 ^{(\alpha)} \text{a.s.} .\]
Then $ \frac{X_{T _n ^{l,X}} \cdot l}{T _n ^{l,X}} \to v \cdot l $ and $  T _n ^{l,X} \sim \frac{n}{v \cdot l}$. It implies that $ \frac{A(T _n ^{l,X})}{n ^{ \frac{1}{\alpha} } } \sim \frac{A(T _n ^{l,X})}{(T _n ^{l,X}) ^{ \frac{1}{\alpha} } }  ( v \cdot l) ^{ \frac{1}{\alpha} }  \to _{n \to +\infty} 0 $ in $\P ^{(\alpha)}$-probability, for all $ \alpha < \kappa $.

It gives $ \lim _{n \to + \infty }  \frac{ \log( A( T _n ^{l,X} ) ) }{ \log (n) }  \leq \frac{1}{\kappa} $ and concludes the proof of the upper bound.

\textbf{Lower bound}

This proof follows the lines of the proof of proposition 12 in \cite{Tournier}. As $\kappa \leq 1$, we can assume that $ \alpha _1 + \alpha _{-1} \geq 2 \sum _{j=1} ^{2d} \alpha _j - 1 $. We prove that, for every $l \in \{ e_1, \dots, e_{2d} \}$, for every $ \alpha  > \kappa $, $ \frac{T _{2n} ^{l,Z}}{ n ^{ \frac{1}{\alpha} } } \to _{n \to \infty}  + \infty $ $ \P ^{(\alpha)} $ a.s.. The same being true for $( T _{2n+1} ^{l,Z} )$, this is sufficient to conclude. 

Set $ l \in \{ e_1, \dots, e_{2d} \}$. We introduce the exit times 
\[ \Theta _0 = \inf \{  n \in \N | Z_n \notin \{ Z_0 , Z_0 + e_1 \}  \}  \]
(with a minus sign instead of the plus if $l = -e_1$), and for $ k \geq 1 $, $ \Theta _k = \Theta _0 \circ \tau _{ T _{2k} ^{l,Z} } $ (where $\tau$ is the time-shift). We use the convention that $ \Theta _k = \infty $ if $ T _{2k} ^{l,Z} = \infty $. The only dependence between the times $ \Theta _k $ is that $ \Theta _j = \infty $ implies $ \Theta _k = \infty $ for all $ k \geq j $. The $"2"$ in $T _{2k} ^{l,Z}$ causes indeed $ \Theta _k $ to depend only on $ \{ x \in \Z ^d | x \cdot l \in \{ 2k, 2k+1 \}  \} $ which are disjoint parts of the environment.

For $ t_0, \dots, t_k \in N $, one has, using the Markov property at time $ T _{2k} ^{l,Z} $, the independence and the translation invariance of $ \P ^{(\alpha)} $ :
\begin{align*}
  \P _0 ^{ (\alpha) } \(  \Theta _0 = t_0, \dots , \Theta _k = t_k \)
 	&=   \P _0 ^{ (\alpha) } \(  \Theta _0 = t_0, \dots , \Theta _{k-1} = t_{k-1}, \Theta _k = t_k, T _{2k} ^{l,Z} < \infty \) &\\
	&\leq  \P _0 ^{ (\alpha) } \(  \Theta _0 = t_0, \dots , \Theta _{k-1} = t_{k-1} \)  \P _0 ^{ (\alpha) }  \( \Theta _0 = t_k \) &\\
    &\leq  \dots  \leq  \P _0 ^{ (\alpha) }  \( \Theta _0 = t_0 \) \dots \P _0 ^{ (\alpha) }  \( \Theta _0 = t_{k - 1} \) \P _0 ^{ (\alpha) }  \( \Theta _0 = t_k \) &\\
 &=  \P  ^{ (\alpha) } \(  \hat{\Theta} _0 = t_0, \dots , \hat{\Theta} _k = t_k \)  &
 \end{align*}
where, under $ \P ^{(\alpha)} $, the random variables $ \hat{\Theta} _k $  are independent and have the same distribution as $ \Theta _0 $. From this, we deduce that for all $ A \subset \N ^\N $, 
\[  \P _0 ^{ (\alpha) } \( ( \Theta _k ) \in A \)  \leq \P  ^{ (\alpha) } \( ( \hat{\Theta} _k ) \in A \) .\]
In particular, for  $ \alpha  > \kappa $, 
\begin{equation}
\label{TLLS}
\P _0 ^{ (\alpha) } \(  \liminf _k  \frac{ \Theta _0 +  \dots + \Theta _{k-1} }{ k ^{ \frac{1}{\alpha} } }  < \infty    \)  \leq \P  ^{ (\alpha) } \(   \liminf _k  \frac{ \hat{\Theta} _0 +  \dots + \hat{\Theta} _{k-1} }{ k ^{ \frac{1}{\alpha} } }  < \infty     \)   .
\end{equation}

In order to bound this probability, we compute the tail of the distribution of $\Theta _0$ using Stirling's formula : 
\begin{align*}
  \P _0 ^{ (\alpha) } \(  \Theta _0 \geq n \)
 	&=   \E \(  \omega( 0, e_1 ) ^{ \ceil{\frac{n}{2}} }   \omega( e_1, 0 ) ^{ \floor{\frac{n}{2}} } \)   &\\
	&=   \frac{ \Gamma(\alpha_0) ^2 }{ \Gamma(\alpha_1) \Gamma(\alpha_{-1}) }    \frac{ \Gamma(\alpha_1 + \ceil{\frac{n}{2}}) \Gamma(\alpha_{-1} + \floor{\frac{n}{2}} ) }{ \Gamma(\alpha_0 + \ceil{\frac{n}{2}})  \Gamma(\alpha_0 + \floor{\frac{n}{2}}) }    &\\
    &\sim _{n \to \infty}      c  n ^{ \alpha _1 + \alpha _{-1} - 2 \alpha _0 }   =  c n ^{ - \kappa }      &
 \end{align*}
with $c$ a constant. We can then use the limit theorem for stable laws (see for example \cite{D}) that gives : 
\[ \frac{ \hat{\Theta} _0 +  \dots + \hat{\Theta} _{k-1} }{ k ^\frac{1}{\kappa} } \Rightarrow  Y \]
where $Y$ has a non-degenerate distribution. Then for $ \alpha > \kappa $, $\frac{ \hat{\Theta} _0 +  \dots + \hat{\Theta} _{k-1} }{ k ^\frac{1}{\alpha} } \to \infty  $. (\ref{TLLS}) then gives $ \P _0 ^{ (\alpha) } \(  \liminf _k  \frac{ \Theta _0 +  \dots + \Theta _{k-1} }{ k ^{ \frac{1}{\alpha} } }  < \infty    \)  = 0  $.

As $ T _{2k} ^{l,Z}  \geq \Theta _0 +  \dots + \Theta _{k-1} $, it gives  $ \frac{T _{2n} ^{l,Z}}{ n ^{ \frac{1}{\alpha} } } \to _{n \to \infty}  + \infty $ $ \P ^{(\alpha)} $ a.s. as wanted, for all $ \alpha > \kappa $. It gives $ \lim _{n \to + \infty }  \frac{ \log(  T _n ^{l,Z}  ) }{ \log (n) }  \geq \frac{1}{\kappa} $ and concludes the proof of the lower bound.

\end{proof}

Using an inversion argument, we can now prove theorem \ref{theoremelimite}.

\begin{proof}[Proof of theorem \ref{theoremelimite}]
We note $ \overline{Z_n} = \max _{i \leq n} Z_i \cdot l $. As $ \overline{Z_n} \geq m \Leftrightarrow  T _m ^{l,Z} \leq n $, theorem \ref{tpsatteinte} gives that for any $\epsilon > 0$ we have, for $n$ big enough,
\[  n ^{\kappa - \epsilon} \leq  \overline{Z_n}  \leq  n ^{\kappa + \epsilon} \text{ in } \P ^{(\alpha)} \text{-probability}.\]

As $ Z_n \cdot l $ is transient, we can introduce renewal times $ \tau _i $ for the direction $l$ (see \cite{SZ} or \cite{Zeitouni} p71 for a detailed construction) such that $ \tau _i  < + \infty $ $ \P ^{(\alpha)} $ a.s., for all $i$. Then 
\[ 0 \leq \overline{Z_n} - Z_n \cdot l  \leq \max _{i=0, \dots, n-1}  ( Z _{ \tau _{i+1} } - Z _{\tau_i}  ) \cdot l \; \text{ for } n \geq \tau_1 .\]

When the walk $ Z_n \cdot l $ discovers a new  vertex in direction $l$, there is a positive probability that this vertex will be the next $ Z _{\tau _i} $. As the vertexes have  i.i.d. exit probabilities under $ \P ^{(\alpha)} $, this probability is independent of the newly discovered vertex, and is independent of the path that lead to this vertex. Then $ ( Z _{ \tau _{i+1} } - Z _{\tau_i}  ) \cdot l $ follows a geometric law of parameter $ \P ^{(\alpha)} ( Z_0 = Z_{\tau _1} ) $, for all $ i \in \N $. This means that we can find $C$ and $c$ two positive constants such that for all $n$, $ \P ^{(\alpha)} \( ( Z _{ \tau _{i+1} } - Z _{\tau_i}  ) \cdot l  \geq  n  \) \leq  C e^{ - c n } $.

Borel Cantelli's lemma then gives that, for $n$ big enough, 
\[ \max _{i=0, \dots, n-1}  ( Z _{ \tau _{i+1} } - Z _{\tau_i}  ) \cdot l  \leq  (\log n)^2  \quad  \P ^{(\alpha)} \text{ a.s.}.\]
As $ \tau _1 < \infty $, it gives 
\[  n ^{\kappa - \epsilon} \leq  Z_n \cdot l  \leq  n ^{\kappa + \epsilon} \text{ in } \P ^{(\alpha)} \text{-probability}.\]
Taking the limit $ \epsilon \to 0 $ gives $ \lim _{n \to + \infty} \frac{ \log( Z _n \cdot l ) }{ \log (n) } = \kappa $ and concludes the proof.

\end{proof}

\appendix
\section{Proof of lemma \ref{finiteexpectation}}

The proof that follows is largely inspired by the article \cite{Tournier} by Tournier. His result can however not be directly applied here, as $ \gamma ^\omega (x) \geq G ^{\omega,\Lambda} (x,x) $, and some of the paths he considered are not necessarily simple paths. To adapt the proof to our case, we need an additional assumption on the graph (some symmetry property for the edges), which simplifies the proof (the construction of the set $C(\omega)$ is quite shorter).

To prove the result, we consider the case of finite directed graphs with a cemetery vertex. A vertex $\delta$ is said to be a cemetery vertex when no edge exits $\delta$, and every vertex is connected to $\delta$ through a directed path. We furthermore suppose that the graphs have no multiple edges, no elementary loop (consisting of one edge starting and ending at the same point), and that if $ (x,y) \in E $ and $ y \neq \delta $, then $ (y,x) \in E $.

We need a definition of $\gamma ^\omega (x)$ for those graphs. Let $G = (V \cup \{ \delta \} ,E)$ be a finite directed graph, $ ( \alpha (e) ) _{e \in E} $ be a family of positive real numbers, $ \P ^{(\alpha)} $ be the corresponding Dirichlet distribution, and $(Z_n)$ the associated random walk in Dirichlet environment. We need the following stopping times : the hitting times 
\[  H_x = \inf \{ n \geq 0  | Z_n = x \} \]
 and 
\[ \tilde{H}_x = \inf \{ n \geq 1  | Z_n = x \} \]
for $x \in G$, the exit time 
\[ T_A = \inf \{ n \geq 0 | Z_n \notin A  \} \] 
for $ A \subset V $, and the time of the first loop 
\[ L = \inf \{ n \geq 1  | \exists n_0 < n \text{ such that } Z_n = Z_{n_0} \}. \]
For $x$ in such a $G$, we define :
\[ \gamma ^\omega (x) = \frac{1}{P_x ^\omega ( H _\delta < \tilde{H}_x \wedge L  )} = \frac{1}{\Sum_ {\sigma : x \to \delta} \omega _\sigma} .\]
where we sum on simple paths from $x$ to $\delta$. In the following, we denote by $0$ an arbitrary fixed vertex in $G$. We use the notations $ \underline{A} = \{ \underline{e} | e \in A \} $ and $ \overline{A} = \{ \overline{e} | e \in A \} $ for $ A \subset E $, and we call strongly connected a subset $A$ of $E$ such that for all $ x, y \in \overline{A} \cup \underline{A} $, there is a path in $A$ from $x$ to $y$. Remark that if $A$ is strongly connected, then $ \overline{A} = \underline{A}$.

For the new function $\gamma ^\omega$ on $G$, we get the following result 

\begin{theo}
\label{finiteexpectationfinitegraph}
Let $ G = ( V \cup \{ \delta \} , E ) $ be a finite directed graph, where $ \delta $ is a cemetery vertex. We furthermore suppose that $G$ has no multiple edges, no elementary loop, and that if $ (x,y) \in E $ and $ y \neq \delta $, then $ (y,x) \in E $. Let $ ( \alpha (e) ) _{e \in E} $ be a family of positive real numbers, and $ \P ^{(\alpha)} $ be the corresponding Dirichlet distribution. Let $ 0 \in V $. There exist $c,C,r > 0$ such that, for $t$ large enough, 
\[ \P ^{(\alpha)} ( \gamma ^\omega (0) > t ) \leq C \frac{ ( \ln t )^r }{ t ^{ \min _A \beta _A } }  \]
where the minimum is taken over all strongly connected subsets $A$ of $E$ such that $ 0 \in \underline{A} $, and $ \beta _A = \sum _{e \in \partial_+ \underline{A} } \alpha (e)$, (we recall that $\partial_+(K) = \{ e \in E, \; \underline{e} \in K , \; \overline{e} \notin K \}$).
\end{theo} 

In $ \Z ^d $, we can identify $ \Lambda ^c $ (where $\Lambda$ is the subset involved in the construction of $\gamma ^\omega$) with a cemetery vertex $\delta$. We obtain a graph where the two definitions of $\gamma ^\omega$ coincide, and that verifies the hypothesis of theorem \ref{finiteexpectationfinitegraph}. Among the strongly connected subsets $A$ of edges such that $ \underline{A} $ contains a given $x$, the ones minimizing the "exit sum" $\beta _A$ are made of only two edges $(x,x+e_i)$ and $(x+e_i,x)$, $ i \in [\!| 1, 2d |\!] $. Then $ \min _A \beta _A  = \kappa = 2 \( \sum_{i=1} ^{2d} \alpha _i \) - \max _{i=1,\dots, d} (\alpha _i + \alpha _{i+d}) $. It proves lemma \ref{finiteexpectation}.

\begin{proof}[Proof of theorem \ref{finiteexpectationfinitegraph}]

This proof is based on the proof of the "upper bound" in \cite{Tournier}. We need lower bounds on the probability to reach $\delta$ by a simple path. We construct a random subset $C(\omega)$ where a weaker ellipticity condition holds. Quotienting by this subset  allows to get a lower bound for the equivalent of $ P_0 ^\omega ( H _\delta < \tilde{H}_0 \wedge L  ) $ in the quotient graph. Proceeding by induction then allows to conclude.

We proceed by induction on the number of edges of $G$. More precisely, we prove :

\begin{prop}
Let $ n \in \N ^* $. Let $ G = ( V \cup \{ \delta \} , E ) $ be a directed graph possessing at most $n$ edges, and such that every vertex is connected to $ \delta $ by a directed path. We furthermore suppose that $G$ has no multiple edges, no elementary loop, and that if $ (x,y) \in E $ and $ y \neq \delta $, then $ (y,x) \in E $. Let $ (\alpha (e)) _{e \in E} $ be positive real numbers. Then, for every vertex $ 0 \in V $, there exist real numbers $C,r>0$ such that, for small $\epsilon > 0$, 
\[  \P ^{(\alpha)} \( P_0 ^\omega ( H_\delta < \tilde{H}_0 \wedge L  ) \leq \epsilon \) \leq C \epsilon ^\beta (-\ln \epsilon) ^r \]
where $ \beta = \min \{ \beta _A | A \text{ is a strongly connected subset of } V \text{ and } 0 \in \underline{A} \} $.
\end{prop}

As $ \gamma ^\omega (0) = \frac{1}{P_0 ^\omega ( H_\delta < \tilde{H}_0 \wedge L  )} $, this proposition suffices to prove the result. The following is devoted to its proof.

Initialization : if $ |E| = 1 $, the only edge links $0$ to $\delta$, then $P_0 ^\omega ( H_\delta < \tilde{H}_0 \wedge L ) = 1$ and the property is true.

If $ |E| = 2 $, the only possible edges link $0$ to $\delta$, and another vertex $x$ to $\delta$, then $P_0 ^\omega ( H_\delta < \tilde{H}_0 \wedge L ) = 1$ and the property is true.

Let $ n \in \N^* $. We suppose the induction hypothesis to be true at rank $n$. Let $ G = ( V \cup \{ \delta \} , E ) $ be a directed graph with $n +1$ edges, and such that every vertex is connected to $ \delta $ by a directed path. We furthermore suppose that $G$ has no multiple edges, no elementary loop, and that if $ (x,y) \in E $ and $ y \neq \delta $, then $ (y,x) \in E $. Let $ (\alpha (e)) _{e \in E} $ be positive real numbers. To get a "weak ellipticity condition", we introduce the random subset $C(\omega)$ of $E$ constructed as follows :

\textbf{Construction of $C(\omega)$.} Let $\omega \in \Omega$. Let $x$ be chosen for $\omega(0,x)$ to be a maximizer on all $\omega(0,y)$, $ y \sim 0 $. If $ x \neq \delta$, we set
\[C(\omega)=\{ (0,x) ; (x,0) \}.\] 
If $x = \delta$, we set $ C(\omega)=\{ (0,\delta)  \} $. Remark that $C(\omega)$ is well defined as soon as $x$ is uniquely defined, which means almost surely, as there is always a directed path heading to $\delta$.

The support of the distribution of $ \omega \to C(\omega) $ writes as a disjoint union $ \mathcal{C} = \mathcal{C}_0 \cup \mathcal{C}_\delta $ depending whether $ x = \delta $ or not. For $C \in \mathcal{C}$, we define the event
\[ \mathcal{E} _C = \{ C(\omega) = C \} .\]
As $\mathcal{C}$ is finite, it is sufficient to prove the upper bound separately on all events $\mathcal{E} _C$. If $ C \in \mathcal{C} _\delta $, on $ \mathcal{E}_C $, $ P _0 ^\omega (H_\delta < \tilde{H}_0 \wedge L) \geq  P_0 ^\omega ( Z_1 = \delta ) \geq \frac{1}{|E|} $ by construction of $C(\omega)$. Then we have for small $\epsilon > 0$ :
\[ \P ^{(\alpha)} \( P_0 ^\omega ( H_\delta < \tilde{H}_0 \wedge L  ) \leq \epsilon , \mathcal{E}_C \) = 0 \]
In the following, we will therefore work on $\mathcal{E}_C$, when $ C \in \mathcal{C}_0$ (ie when $x \neq \delta$). In this case, $C$ is strongly connected.

\textbf{Quotienting procedure.}
\begin{de}
If $A$ is a strongly connected subset of edges of a graph $ G = (V,E) $, the quotient graph of $G$ obtained by contracting $A \subset E $ to the vertex $\tilde{a}$ is the graph $\tilde{G}$ deduced from $G$ by deleting the edges of $A$, replacing all the vertices of $ \underline{A} $ by one new vertex $\tilde{a}$, and modifying the endpoints of the edges of $E \setminus A$ accordingly. Thus the set of edges of $\tilde{G}$ is naturally in bijection with $E \setminus A$ and can be thought of as a subset of $E$. 
\end{de}

In our case, we consider the quotient graph $\tilde{G}$ obtained by contracting $C(\omega)$, which is a strongly connected subset of $E$, to a new vertex $ \tilde{0} $. We need to define the associated quotient environment $ \tilde{\omega} \in \tilde{\Omega} $. For every edge in $ \tilde{E} $, if $ e \notin \partial _+ \underline{C} $ then $ \tilde{\omega}(e) = \omega (e) $, and if $ e \in \partial _+ \underline{C}  $, $ \tilde{\omega}(e) = \frac{\omega (e)}{\Sigma} $, where $ \Sigma = \sum _{e \in \partial _+ \underline{C} } \omega (e) $. 

This environment allows us to bound $\gamma ^\omega (0)$ using the similar quantity in $ \tilde{G} $. Notice that, from $0$, one way for the walk to reach $\delta$ without coming back to $0$ and without making loops consists in exiting $C$ without coming back to $0$, and then reaching $\delta$ without coming back to $\underline{C}$ ($0$ or $x$) and without making loops. Then, for $ \omega \in \mathcal{E}_C $,
\begin{align*}
& P_0 ^\omega  ( H _\delta < \tilde{H}_0 \wedge L ) &\\
&\geq  P_0 ^\omega ( H_\delta <  \tilde{H} _{\underline{C}} \wedge L  ) + P_0 ^\omega ( Z_1 = x , H_\delta < 1 + ( \tilde{H} _{\underline{C}} \wedge L ) \circ \tau _{1} )  & \\
&= P_0 ^\omega ( H_\delta <  \tilde{H} _{\underline{C}} \wedge L  ) + P_0 ^\omega ( Z_1 = x ) P_x ^\omega ( H_\delta < \tilde{H} _{\underline{C}} \wedge L  )  & \\
&\geq P_0 ^\omega ( H_\delta <  \tilde{H} _{\underline{C}} \wedge L  ) + \frac{1}{|E|}  P_x ^\omega ( H_\delta < \tilde{H} _{\underline{C}} \wedge L  )  & \\
&\geq \frac{1}{|E|} \( P_0 ^\omega ( H_\delta <  \tilde{H} _{\underline{C}} \wedge L  ) +   P_x ^\omega ( H_\delta < \tilde{H} _{\underline{C}} \wedge L  ) \)  & \\
&= \frac{1}{|E|} \Sigma P _{\tilde{0}} ^{\tilde{\omega}} ( H_\delta < \tilde{H} _{\tilde{0}} \wedge L  ) &
\end{align*}
where we used the Markov property, the construction of $C$, $ \frac{1}{|E|} \leq 1 $, and the definition of the quotient. Finally, we have
\begin{equation}
\label{quotient}
\P ^{(\alpha)} \( P_0 ^\omega (H_\delta < \tilde{H}_0 \wedge L ) \leq \epsilon , \mathcal{E}_C  \)  \leq  \P ^{(\alpha)} \( \Sigma P _{\tilde{0}} ^{\tilde{\omega}} ( H_\delta < \tilde{H} _{\tilde{0}} \wedge L  ) \leq |E| \epsilon , \mathcal{E}_C  \) .
\end{equation}

\textbf{Back to Dirichlet environment.}
Under $\P ^{(\alpha)}$, $\tilde{\omega}$ does not follow a Dirichlet distribution because of the normalization. But we can reduce to the Dirichlet situation with the following lemma (which is a particular case of lemma 9 in \cite{Tournier}). 
\begin{lemma}
Let $ ( \omega_i ^{(0)} ) _{1 \leq i \leq n_0},( \omega_i ^{(x)} ) _{1 \leq i \leq n_x}  $ be the exit probabilities out of $0$ and $x$ for $ \omega \in \Omega $, they are independent random variables following Dirichlet laws of respective parameters $ ( \alpha_i ^{(0)} ) _{1 \leq i \leq n_0},( \alpha_i ^{(x)} ) _{1 \leq i \leq n_x}  $. Let $ \Sigma = \sum _{e \in \partial _+ \underline{C} } \omega (e) $ and $ \beta_C = \sum _{e \in \partial _+ \underline{C} } \alpha (e) $. There exists positive constants $c, c'$ such that, for every $\epsilon > 0$,
\[ \P ^{(\alpha)} \( \Sigma P _{\tilde{0}} ^{\tilde{\omega}} ( H_\delta < \tilde{H} _{\tilde{0}} \wedge L  ) \leq \epsilon  \) 
\leq 
c \tilde{\P} ^{(\alpha)}  \(  \tilde{\Sigma}  P _{\tilde{0}} ^\omega ( H_\delta < \tilde{H} _{\tilde{0}} \wedge L  ) \leq \epsilon \) , \]
where $\tilde{\P} ^{(\alpha)}$ is the Dirichlet distribution of parameter $ (\alpha(e)) _{e \in \tilde{E}} $ on $\tilde{\Omega}$, $\omega$ is the canonical random variable on $\tilde{\Omega}$, and, under $\tilde{\P} ^{(\alpha)}$, $\tilde{\Sigma}$ is a positive bounded random variable independent of $ \omega $ and such that, for all $\epsilon > 0$, $ \tilde{\P} ^{(\alpha)} ( \tilde{\Sigma} \leq \epsilon ) \leq c' \epsilon ^{\beta_C} $.
\end{lemma}
Remark that the symmetry property we imposed on the edges is important here : if there was no edge from $x$ to $0$, the probability for a walk in $\tilde{G}$ to exit $\tilde{0}$ through one of the edges exiting $x$ in $G$ would necessarily be bigger than $ \frac{1}{2} $. Then asymptotically, it could not be bounded by Dirichlet variables.

This lemma and (\ref{quotient}) give :
\begin{align}
\label{lienquotient}
\P ^{(\alpha)} \( P_0 ^\omega (H_\delta < \tilde{H}_0 \wedge L ) \leq \epsilon , \mathcal{E}_C  \)
& \leq   \P ^{(\alpha)} \( \Sigma P _{\tilde{0}} ^{\tilde{\omega}} ( H_\delta < \tilde{H} _{\tilde{0}} \wedge L  ) \leq |E| \epsilon , \mathcal{E}_C  \)    & \\ \nonumber
& \leq  \P ^{(\alpha)} \( \Sigma P _{\tilde{0}} ^{\tilde{\omega}} ( H_\delta < \tilde{H} _{\tilde{0}} \wedge L  ) \leq |E| \epsilon  \)   &\\ \nonumber
& \leq c \tilde{\P} ^{(\alpha)}  \(  \tilde{\Sigma}  P _{\tilde{0}} ^\omega ( H_\delta < \tilde{H} _{\tilde{0}} \wedge L  ) \leq |E| \epsilon \)   .&
\end{align}

\textbf{Induction.}
Inequality (\ref{lienquotient}) relates the same quantities in $G$ and $\tilde{G}$, allowing to complete the induction argument.

The edges in $C$ do not appear in $\tilde{G}$ any more : $\tilde{G}$ has $n - 2$ edges. In order to apply the induction hypothesis, we need to check that each vertex is connected to $\delta$. This results directly from the same property for $G$. If $ (x,y) \in \tilde{E} $ and $ y \neq \delta $, then $ (x,y) \notin C(\omega) $ and $ (y,x) \notin C(\omega) $. As only the edges of $C(\omega)$ disappeared, then $ (y,x) \in \tilde{E} $. $\tilde{G}$ has no elementary loop. Indeed $G$ has none, and the quotienting only merges the vertices of $ \underline{C} $, whose joining edges are those of $C$, deleted in the construction. It only remains to prove that  $\tilde{G}$ has no multiple edges. It is not necessarily the case (quotienting may have created multiple edges), but it is possible to reduce to this case, using the additivity property of the Dirichlet distribution.

The induction hypothesis applied to $ \tilde{G} $ and $ \tilde{0} $ then gives, for small $\epsilon > 0$, 
\begin{equation}
\label{induction}
\tilde{\P} ^{(\alpha)}  \(   P _{\tilde{0}} ^\omega ( H_\delta < \tilde{H} _{\tilde{0}} \wedge L  ) \leq  \epsilon \) \leq  c'' \epsilon ^{\tilde{\beta}} (-\ln \epsilon ) ^r  ,
\end{equation}
where $c'' > 0$, $r>0$ and $\tilde{\beta}$ is the exponent "$\beta$" from the statement of the induction hypothesis corresponding to the graph $ \tilde{G} $.

This inequality, associated with (\ref{lienquotient}) and the following simple lemma (also see \cite{Tournier} for the proof of the lemma) then allows to carry out the induction : 

\begin{lemma}
If $X$ and $Y$ are independent positive bounded random variables such that, for some real numbers $ \alpha_X, \alpha_Y, r>0 $,
\begin{itemize}
\item there exists $C>0$ such that $ P (X < \epsilon) \leq C \epsilon ^{\alpha_X} $ for all $\epsilon > 0$ (or equivalently for small $\epsilon$);
\item there exists $C'>0$ such that $ P (Y < \epsilon) \leq C' \epsilon ^{\alpha_Y} (-\ln \epsilon) ^r $ for small $\epsilon > 0$;
\end{itemize}  
then there exists a constant $C''>0$ such that, for small $\epsilon >  0$, 
\[ P (XY \leq \epsilon) \leq C'' \epsilon ^{ \alpha _X \wedge \alpha _Y } (-\ln \epsilon) ^{r+1} \]
(and $r+1$ can be replaced by $r$ if $ \alpha _X \neq \alpha _Y $).
\end{lemma}

We get from this lemma, (\ref{lienquotient}) and (\ref{induction}) some constants $c, r>0$ such that, for small $\epsilon >0$, 
\[  \P ^{(\alpha)} \( P_0 ^\omega (H_\delta < \tilde{H}_0 \wedge L ) \leq \epsilon , \mathcal{E}_C  \)  \leq  c \epsilon^{ \beta_C \wedge \tilde{\beta} } (-\ln \epsilon) ^{r+1}  .\]
It remains to prove that $ \tilde{\beta} \geq \beta $, where $\beta$ is the exponent defined in the induction hypothesis relative to $G$ and $0$. Let $\tilde{A}$ be a strongly connected subset of $\tilde{E}$ such that $ \tilde{0} \in \underline{\tilde{A}} $. Set $ A = \tilde{A} \cup C \subset E $. In view of the definition of $ \tilde{E} $, every edge exiting $ \tilde{A} $ corresponds to an edge exiting $A$, and vice-versa (the only edges deleted in the quotient procedure are those of $C$). Thus, recalling that the weights of the edges are preserved in the quotient, $ \beta _{ \tilde{A} } = \beta _A $. Moreover, $ \tilde{0} \in \underline{A} $ and $A$ is strongly connected, so that $ \beta _A \geq \beta $. As a consequence, $ \tilde{\beta} \geq \beta $ as announced.

Then $ \beta _C  \wedge \tilde{\beta} \geq \beta _C \wedge \beta = \beta $ because $C$ is strongly connected, and $ 0 \in \underline{C}$. It gives, for small $\epsilon >0$ :
\[  \P ^{(\alpha)} \( P_0 ^\omega (H_\delta < \tilde{H}_0 \wedge L ) \leq \epsilon , \mathcal{E}_C  \)  \leq  c \epsilon^{ \beta } (-\ln \epsilon) ^{r+1}  .\]
Summing on all events $ \mathcal{E} _C, C \in \mathcal{C} $ concludes the induction and the proof.

\end{proof}

\subsection*{Acknowledgement}
I would like to thank Christophe Sabot for helpful discussions and suggestions.


\begin{thebibliography}{}




\bibitem{Atkinson} Atkinson, Giles; \textit{Recurrence of co-cycles and random walks}. J. London Math. Soc. (2) 13 (1976), no. 3, 486–488. 
 
\bibitem{Sz} Bolthausen, Erwin; Sznitman, Alain-Sol; \textit{Ten lectures on random media}. DMV Seminar, 32. Birkhäuser Verlag, Basel, 2002. vi+116 pp.

\bibitem{D} Durrett, Richard; \textit{Probability : theory and examples}. Second edition. Duxbury Press, Belmont, CA, 1996. xiii+503 pp.
 
\bibitem{ES} Enriquez, Nathanaël; Sabot, Christophe; \textit{Edge oriented reinforced random walks and RWRE}. C. R. Math. Acad. Sci. Paris 335 (2002), no. 11, 941–946 

\bibitem{ES2} Enriquez, Nathanaël; Sabot, Christophe; \textit{Random walks in a Dirichlet environment}. Electron. J. Probab. 11 (2006), no. 31, 802–817 (electronic).
 
\bibitem{Kalikow} Kalikow, Steven A.; \textit{Generalized random walk in a random environment}. Ann. Probab. 9 (1981), no. 5, 753–768.
 
\bibitem{KeaneRolles} Keane, M. S.; Rolles, \textit{S. W. W. Tubular recurrence}. Acta Math. Hungar. 97 (2002), no. 3, 207–221. 
 
\bibitem{KKS} Kesten, H.; Kozlov, M. V.; Spitzer, F.; \textit{A limit law for random walk in a random environment}. Compositio Math. 30 (1975), 145–168.  
 
\bibitem{K} Krengel, Ulrich; \textit{Ergodic theorems}. With a supplement by Antoine Brunel. de Gruyter Studies in Mathematics, 6. Walter de Gruyter \& Co., Berlin, 1985. viii+357 pp.
 
\bibitem{LP} Lyons, Russell; Peres, Yuval; \textit{Probabilities on trees and networks}. Cambridge University Press. In preparation, available at  http://mypage.iu.edu/\string~rdlyons/ .
 
\bibitem{Pemantle} Pemantle, Robin; \textit{Phase transition in reinforced random walk and RWRE on trees}. Ann. Probab. 16 (1988), no. 3, 1229–1241. 
 
 \bibitem{S2} Sabot, Christophe; \textit{Random walks in random Dirichlet environment are transient in dimension $ d \geq 3 $}. Probab. Theory Related Fields 151 (2011), no. 1-2, 297-317.
 
 \bibitem{S} Sabot, Christophe; \textit{Random Dirichlet environment viewed from the particle in dimension $ d \geq 3 $}. To appear in Annals of Probability.
 
 \bibitem{ST} Sabot, Christophe; Tournier, Laurent; \textit{Reversed Dirichlet environment and directional transience of random walks in Dirichlet environment}. Ann. Inst. Henri Poincaré Probab. Stat. 47 (2011), no. 1, 1–8
 
 \bibitem{Sinai} Sinaï, Ya. G. \textit{The limit behavior of a one-dimensional random walk in a random environment}. (Russian) Teor. Veroyatnost. i Primenen. 27 (1982), no. 2, 247–258. 
 
 \bibitem{Solomon} Solomon, Fred; \textit{Random walks in a random environment}. Ann. Probability 3 (1975), 1–31.
 
 \bibitem{SZ} Sznitman, Alain-Sol; Zerner, Martin; \textit{A law of large numbers for random walks in random environment}. Ann. Probab. 27 (1999), no. 4, 1851–1869. 
 
 \bibitem{Tournier} Tournier, Laurent; \textit{Integrability of exit times and ballisticity for random walks in Dirichlet environment}. Electron. J. Probab. 14 (2009), no. 16, 431–451.
 
 \bibitem{Tournier2} Tournier, Laurent; \textit{A note on directional transience and asymptotic direction of random walks in Dirichlet environment}. In preparation.
 
  \bibitem{Zeitouni} Zeitouni, Ofer; \textit{Random walks in random environment}. Lectures on probability theory and statistics, 189–312, Lecture Notes in Math., 1837, Springer, Berlin, 2004.
 
 \bibitem{ZM} Zerner, Martin P. W.; Merkl, Franz; \textit{A zero-one law for planar random walks in random environment}. Ann. Probab. 29 (2001), no. 4, 1716–1732. 
 

\end{thebibliography}
\end{document}